\documentclass[twoside,11pt]{article}

\usepackage{amsfonts,amsmath,amssymb,enumerate}
\usepackage[dvips]{graphicx}
\usepackage{psfrag,graphicx}
\usepackage{amsthm}

\usepackage{color,graphics}
\usepackage{graphics}
\usepackage{epsfig}
\usepackage{subfigure}

\newtheorem{theorem}{Theorem}[section]
\newtheorem{lemma}[theorem]{Lemma}

\numberwithin{equation}{section}

\def\Vo{\vbox{\offinterlineskip\hbox{\kern 3pt$\scriptstyle\circ$}
\kern 1pt\hbox{$V$}}}

\def\Ho{\vbox{\offinterlineskip\hbox{\kern 3pt$\scriptstyle\circ$}
\kern 1pt\hbox{$H$}}}

\def\Wo{\vbox{\offinterlineskip\hbox{\kern 3pt$\scriptstyle\circ$}
\kern 1pt\hbox{$W$}}}

\newcommand{\bt}{\begin{theorem}}
\newcommand{\et}{\end{theorem}}

\def\eqldef{\overset{\text{\tiny \rm def}}{=}}
\newcommand{\norm}[2][{}]{\lVert#2\rVert_{{#1}}}
\def\tra{\mathsf{T}}

\newcommand{\abs}[2][{}]{\lvert#2\rvert_{#1}}
\def\dd{\;\!\mathrm{d}}    
\def\C{{\mathrm{C}}}
\def\H{{\mathrm{H}}}
\def\L{{\mathrm{L}}}
\def\V{{\mathrm{V}}}

\def\calK{{\mathcal{K}}}
\def\calV{{\mathrm{V}}}
\def\calH{{\mathrm{H}}}

\def\bproof{\begin{proof}}
\def\eproof{\end{proof}}

\definecolor{green1}{rgb}{0.1,0.65,0}

\begin{document}

\title{A weighted finite element mass redistribution method\\ for 
dynamic contact problems}
\date{}
\author{F.~Dabaghi\thanks{Universit\'e de Lyon, CNRS, INSA-Lyon, 
Institut Camille Jordan UMR 5208,
20 Avenue A. Einstein, F-69621 Villeurbanne, France 
(\tt farshid.dabaghi@insa-lyon.fr, apetrov@math.univ-lyon1.fr, jerome.pousin@insa-lyon.fr, 
Yves.Renard@insa-lyon.fr)}\and
P.~Krej{\v{c}}{\'{\i}}\thanks{Institute of Mathematics, Czech Academy of Sciences,
\v{Z}itn\'{a} 25, CZ-11567 Praha 1, Czech Republic, {\tt krejci@math.cas.cz}.}
\and A.~Petrov$^{*}$
		\and J.~Pousin$^{*}$
		\and Y.~Renard$^{*}$}

\pagestyle{myheadings}
\thispagestyle{plain}
\markboth{F. Dabaghi, P. Krej{\v{c}}{\'{\i}}, A. Petrov,  J. Pousin, Y. Renard}
{ \footnotesize{A weighted finite element mass redistribution method for 
dynamic contact problems}}

\maketitle

\begin{abstract}
This paper deals with a one-dimensional wave equation being subjected to a 
unilateral boundary condition. An approximation of this problem 
combining the finite element and mass redistribution methods is proposed.
The mass redistribution method is based on a redistribution of the body mass 
such that there is no inertia at the contact node and the mass of the contact node
is redistributed on the other nodes. The convergence 
as well as an error estimate
in time are proved. The analytical solution associated  with
a benchmark problem is introduced and it is compared to approximate solutions
for different choices of mass redistribution.
However some oscillations for the energy associated with approximate solutions obtained 
for the second order schemes can be observed after the impact. To overcome this difficulty, 
an new unconditionally stable and a very lightly dissipative scheme
is proposed. 
\end{abstract}

\hspace*{-0.6cm}{\bf {Key words.}}
Numerical solution, mass redistribution method, variational
inequality, unilateral contact, energy conservation.
\vspace{0.3cm}

\hspace*{-0.6cm}{\bf {AMS Subject Classification.}}
35L05, 35L85, 49A29, 65N10, 65N30, 74M15.


\section{Introduction}
\label{desc_models}
\label{desc_models_ch3}

The present paper highlights some new numerical results obtained for a one-dimensional
elastodynamic contact problem. Dynamical contact problems play a crucial role in 
structural mechanics as well as in biomechanics and a considerable amount of 
engineering and mathematical literature has been dedicated to this topic last decades.
One of the main difficulties in the numerical treatment of such problems is 
the physically meaningful non-penetration condition that is usually modeled by
using the so-called Signorini boundary condition. Basically, the lack of well-posedness 
results mainly originates from the hyperbolic structure of the problem
which gives rise to shocks 
at the contact interfaces. Then the resulting 
nonsmooth and nonlinear variational inequalities lead to fundamental difficulties
in mathematical analysis as well as in the development of numerical integration schemes. 
In view to avoid these difficulties, the non-penetration condition is quite often relaxed
in the numerical integration schemes. We may also observe that 
most of unconditionally stable schemes for the 
linear elastodynamic problems lose their unconditional stability in the presence 
of contact conditions.
Among them the classical Newmark method is the most popular one. However, its
unsatisfactory handling of the non-penetration conditions may lead to artificial 
oscillations at the contact boundary and even give rise to an undesirable energy blow-up
during the time integration, the reader is referred to~\cite{Renard, DOER11, daba13} as well as to
the references therein for further details. 
To overcome these
difficulties, some numerical methods based on the Newmark scheme for
solving impact problems are proposed  in~\cite{Car91}. However, these methods lead to
some important energy losses when the contact takes place even if the time
step is taken sufficiently small. On the other hand, the energy conserving time integration schemes 
of Newmark type are introduced in~\cite{LauLov02, LauCha97, LauCH98} as well as in
the monograph~\cite{LauCCIM03},
but these schemes are unable to circumvent the undesirable oscillations at the contact boundary.
These unphysical oscillations
are avoided by the numerical methods developed in~\cite{DKE07CSNM}
but these methods are still energy
dissipative.
\vspace{0.3em}

Another approach consists in removing the mass at the contact nodes 
and it was originally investigated
in~\cite{Renard} and later on used
in~\cite{HHW08CAFC,Haur10,Ren10,CHHIRE1-14}.
This approach prevents the oscillations at the contact boundary 
and leads to well-posed and energy conserving 
semi-discretization of elastodynamic contact problems (see~\cite{Ligursky2011,DPPR13}).
However, some numerical experiments, 
exhibited in~\cite{daba13}, highlight a phase shift in time
between analytical and approximate solutions. Note that an analytical 
piecewise affine and periodic solution to our problem can be obtained by 
using the characteristics method while approximate solutions are exhibited 
for different time discretizations.
This phase shift in time
comes from the removed mass at the contact nodes for approximate problems which is
unacceptable for many applications.
Therefore a variant of the mass redistribution method is proposed in this work.
More precisely, this new method
consists in transferring the mass of the contact node on the other nodes meaning that
the total mass of the considered material is preserved.  
Numerical experiments presented
in this work show that the undesirable phase shift between
the approximate and analytical solutions disappears and all the 
properties of the mass redistribution method mentioned above are preserved. They highlight  
that the \emph{weighted mass redistribution method} is particularly well adapted to 
deal with contact problems.

\vspace{0.3em}
The paper is organized as follows. 
In Section~\ref{math_form}, the mathematical formulation of a one dimensional elastodynamic
contact problem is presented. The contact is modeled by using the Signorini
boundary conditions in displacement, which are based on a linearization of the physically
meaningful non penetrability of the masses.
Then a space semi-discretization based on a variant of the mass redistribution method
 is presented in Section~\ref{Discretization_ch3}.
This variant of the mass redistribution method consists in transferring the mass of the contact node 
on the other nodes while the inertia vanishes at the contact node. 
The error estimate in time as well as the convergence result are established.  
A benchmark problem is introduced 
in Section \ref{waveeq_ch3} and its analytical solution is exhibited. Then 
numerical experiments for some space-time discretizations like the Crank-Nicolson
or the backward Euler methods are reported. 
These numerical experiments highlight that the choice of the nodes where 
the mass is transferred plays a crucial role to get a better approximate solution.
However some oscillations for the energy associated with approximate solutions 
for the second order schemes like Crank-Nicolson scheme can be observed after the impact.
To overcome the difficulty, a hybrid scheme mixing the Crank-Nicolson as well as
the midpoint methods and having the properties to be 
an unconditionally stable scheme
is proposed in Section \ref{hybrid_scheme}.

\section{Mathematical formulation}
\label{math_form}

The motion of an elastic bar of length $L$ which is free to move as long as it 
does not hit a material obstacle is studied, see Figure~\ref{contact_ch3}. 
The assumptions of small deformations are assumed and the material of the bar
is supposed to be homogeneous. Let $u(x,t)$ be the displacement at time $t\in [0,T]$, \(T>0\) 
of the material point of spatial coordinate $x\in [0,L]$.  Let \(f(x,t)\)
denotes a density of external forces, depending on time and space.  The mathematical
problem is formulated as follows:
\begin{equation}
 \label{wave_ch3}
 u_{tt}(x,t)-u_{xx}(x,t)=f(x,t),\quad (x,t)\in (0,L)\times(0,T),
\end{equation}
with Cauchy initial data
\begin{equation}
 \label{init_cond_ch3}
 u(x,0)=u^0(x) \quad \text{and} \quad u_t(x,0)=v^0(x), \quad x\in (0,L),
\end{equation}
and Signorini and Dirichlet boundary conditions at \(x=0\) and \(x=L\), respectively,  
\begin{equation}
\label{bound_cond_ch3}
0\leq u(0,t)\perp u_x(0,t)\leq 0\quad\text{and}\quad
u(L,t)=0, \quad t\in [0,T].
\end{equation}
Here $u_t\eqldef \frac{\partial u}{\partial t}$ and $u_x\eqldef
 \frac{\partial u}{\partial x}$.  
The orthogonality has a natural meaning: 
an appropriate duality product between two terms of relation vanishes.
It can be alternatively stated as the inclusion
\begin{equation}\label{e1}
u_x(0,t) \in \partial I_{[0,\infty)} (u(0,t)),
\end{equation}
where $I_{[0,+\infty)}$ is
the indicator function of the interval $[0,+\infty)$, and 
$\partial I_{[0,\infty)}$ is its subdifferential.
\begin{figure}[htbp]
\begin{center}
\includegraphics[width=6cm,angle=0]{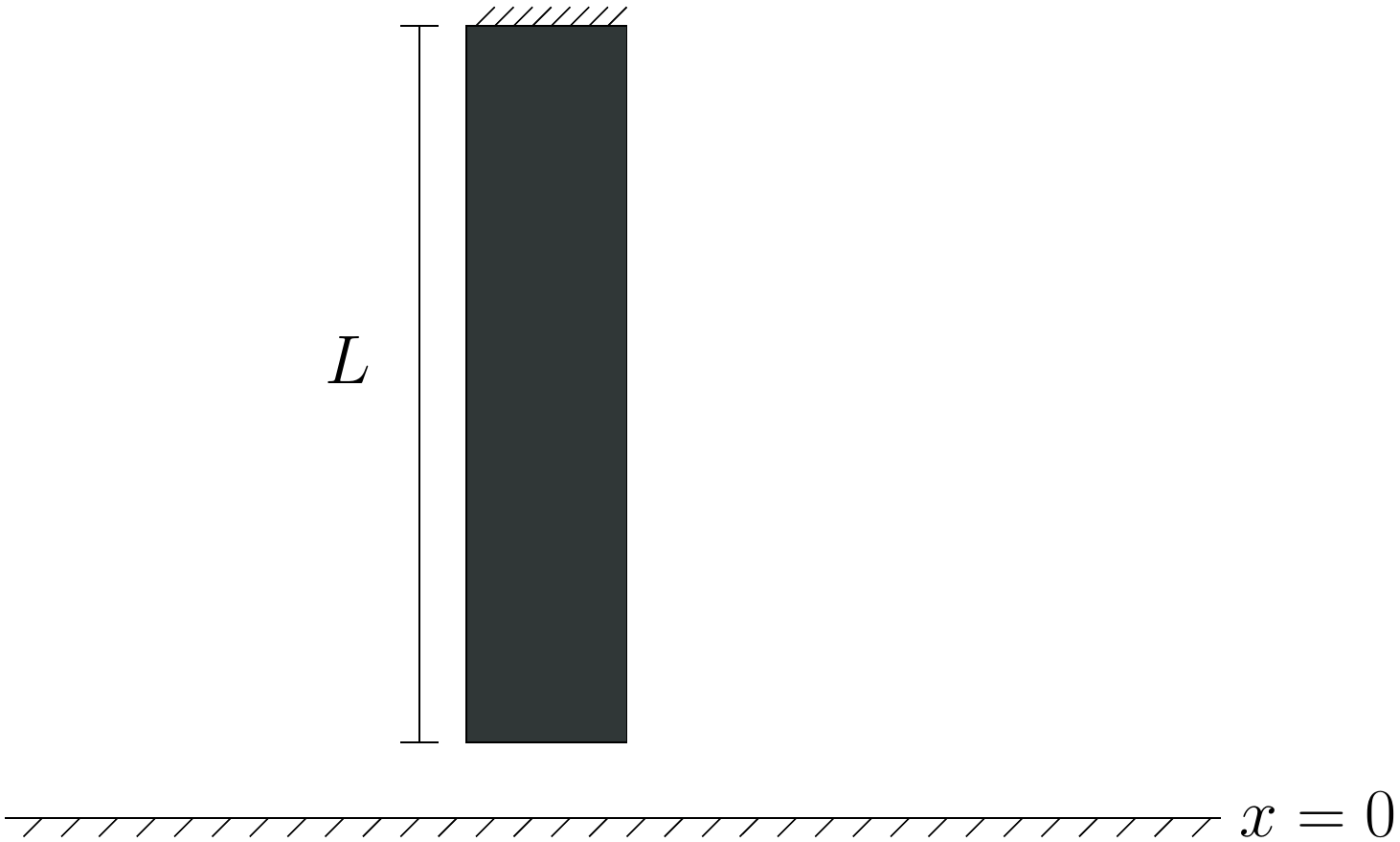}
\end{center}
\caption{An elastic bar vibrating on impacting obstacle.}
\label{contact_ch3}
\end{figure}  

\vspace{0.5em}
Let us describe the weak formulation associated with
\eqref{wave_ch3}--\eqref{bound_cond_ch3}. 
For that purpose, it is convenient to introduce the following notations:
\(\calV\eqldef \{u\in\H^1(0,L): u(L)=0\}\), \(\calH\eqldef\L^2(0,L)\),
\(\mathcal{V}\eqldef \{u\in\L^2(0,T;\calV):  u_t\in\L^2(0,T;\calH)\}\)
and the convex set  
\(\calK\eqldef \{u\in\mathcal{V}:  u(0,\cdot)\geq 0 \text{ a.\,e.}\}\). 
Thus the weak formulation associated with
\eqref{wave_ch3}--\eqref{bound_cond_ch3} 
obtained by multiplying \eqref{wave_ch3} by $v-u$ and by integrating
formally this result over $Q_{T}\eqldef (0,L)\times(0,T)$ reads:
\begin{equation}
\label{wf_ch3}
\begin{cases}
 \text{Find } u\in 
 \calK \text{ such that}\\
 \displaystyle{
 -\int^L_0 v^0(x)(v(x,0){-}u^0(x))\dd x
 - \int_{Q_T} u_t(x,t) ( v_t(x,t){-} u_t(x,t))\dd x\dd t}+\\
\displaystyle{\int_{Q_T} u_x(x,t)(v_x(x,t){-} u_x(x,t))\dd x\dd t\geq 
  \int_{Q_T} f(x,t) (v(x,t){-} u(x,t))\dd x\dd t}\\
 \text{ for all }v\in \calK\text{ for which there exists }\zeta>0
 \text{ with }  v=u\text{ for }t\geq T-\zeta.
 \end{cases}
\end{equation}
For Problem \eqref{wf_ch3}, the following existence and uniqueness result
was proved in~\cite[Theorem 14]{LebSch84}.

\begin{theorem}\label{t1}
Let $u^0 \in \mathrm{H}^{3/2}(0,L)\cap \emph{V}$, 
$v^0 \in \mathrm{H}^{1/2}(0,L)$, $f \in \mathrm{H}^{3/2}(Q_T)$
be given. Then there exists a unique solution $u 
\in \mathrm{L}^\infty(0,T; \mathrm{H}^{3/2}(0,L)\cap V)
\cap \mathrm{W}^{1,\infty}(0,T; \mathrm{H}^{1/2}(0,L))$ of Problem \eqref{wf_ch3}
and the energy balance equation
\begin{equation}
\label{eq:energy2_ch3}
\int_0^L(\abs{u_t(x,\tau)}^2 + \abs{u_x(x,\tau)}^2)\dd x=
\int_0^L(\abs{v^0(x)}^2 + \abs{{u_x^{0}}(x)}^2)\dd x + 2\int_{Q_{\tau}}f(x,t)u_t(x,t)\dd x\dd t
\end{equation}
holds for all $\tau \in [0,T]$.
\end{theorem}

Existence and uniqueness results are obtained for a similar situation of a 
vibrating string with concave obstacle in one dimensional space in \cite{Scha80} and 
also for a wave equation with unilateral constraint at the boundary in a half-space 
of $\mathbb{R}^N$  in~\cite{LebSch84}.
An existence result for a wave equation in a $\C^2$-regular bounded domain 
constrained by an obstacle at the boundary in $\mathbb{R}^2$ 
is proven in \cite{Kim89}. The reader is also referred to~\cite{DPPR13}.


\section{Finite element discretization and convergence of the mass redistribution method}
\label{Discretization_ch3}

This section is devoted to semi-discrete problems in space associated with \eqref{wf_ch3}
by using the mass redistribution method, see \cite{Renard,DPPR13},
assuming that the hypotheses of Theorem \ref{t1} are satisfied.
More precisely, the weighted mass redistribution method consists in transferring  
the mass of the contact node on the other nodes implying that 
the node at the contact boundary  evolves in a quasi-static way.
To this aim, we choose an integer $m > 1$, and put
$h\eqldef \frac{L}{m}$ (mesh size) with the goal to let
 \(m\) tend to \(+\infty\).
We introduce the spaces
\(\calV_h\eqldef  \{ v_h\in \mathrm{C}^0([0,L]) : v_h|_{[ih ,(i{+}1)h]} \in P_1,   
 i=0,\ldots, m-1, v_h(L)= 0\}\) where $P_1$ is the space 
of polynomials of degree less than or equal to 1.
We consider the following discretized problem:
\begin{equation*}
(\mathrm{P}_{u_h})\hspace{1em}
\begin{cases}
\text{Find }  u_h : [0,T] \rightarrow \calV_h  \text{ and }  \lambda_h: 
[0,T] \rightarrow \mathbb{R} \text{ such that for all } v_h\in \calV_h\\
\displaystyle{
\int_0^L \big ((u_{h,tt}{-}f) v_h  {+}   u_{h,x}v_{h,x} \big) \dd x  
= -\lambda_h v_h(0)},\\
\lambda_h(t) \in \partial I_{[0,\infty)} (u_h(0,t)),\\
u_h(\cdot,0)= u_h^0 \quad\text{ and }\quad u_{h,t}(\cdot,0)=v_h^0,
 \end{cases}
\end{equation*}
where \(u^0_h\) and \(v^0_h\) belong to \(\calV_h\) 
and they are the approximations of the initial displacement and velocity \(u_0\)
and \(v_0\), respectively,
and \(\lambda_h\) is the Lagrange multiplier representing the contact force.
The inclusion in \((\mathrm{P}_{u_h})\) (cf.~also \eqref{e1})
can be written as a variational inequality in the form  
\begin{equation*}
\begin{cases}
\lambda_h(t) (u_{h}(0,t){-}z) \geq 0  \quad\text{ for all }\quad z\geq 0\\
u_{h}(0,t)\geq 0.
\end{cases}
\end{equation*}
The approximation $u_{h}$ are taken in the following form
\begin{equation*}
 u_{h}(x,t)= \sum_{k=0}^{m-1}u_k(t)\varphi_k(x),
\end{equation*}
where the basis functions $\varphi_k$ are assumed piecewise linear, namely we have 
\begin{equation*}
\varphi_0 (x) \eqldef
\begin{cases}
1-\frac{x}{h}&\text{ if }x\in [0,h)\\
0 &\text{ if }x \geq h
\end{cases}
\quad\text{ and }\quad
\varphi_k (x) \eqldef
\begin{cases}
\frac{x}{h}-k+1&\text{ if } x\in [(k{-}1)h,kh)\\
k+1-\frac{x}{h}&\text{ if }x\in [kh,(k{+}1)h)\\
0  &\text{ otherwise }\\
\end{cases}
\end{equation*}
 for $k=1,\ldots, m-1$. Notice that \(u_k(t) = u_h(kh,t)\) for \(k=0,\ldots,m-1\) and $t\in [0,T]$.
 The test functions $v_h$ are also considered in the form 
 \begin{equation*}
 v_{h}(x,t)= \sum_{k=0}^{m-1}v_k(t)\varphi_k(x).
\end{equation*}
It is convenient for numerical computations to redistribute the mass and 
modify the problem $(\mathrm{P}_{u_h})$ as follows:
\begin{equation*}
(\mathrm{P}_{u_h}^\textrm{mod})\hspace{1em}
\begin{cases}
\text{Find }  u_h : [0,T] \rightarrow \calV_h  \text{ and }  \lambda_h: 
[0,T] \rightarrow \mathbb{R} \text{ such that for all } v_h\in \calV_h\\
\displaystyle{
\int_0^L \big((u_{h,tt}{-}f) v_h w_h {+}  u_{h,x}v_{h,x} \big) \dd x  
= -\lambda_h v_h(0)},\\
\lambda_h(t)\in\partial I_{[0,\infty)}(u_h(0,t)),\\
u_h(\cdot,0)= u_h^0 \quad\text{ and }\quad u_{h,t}(\cdot,0)=v_h^0,
 \end{cases}
\end{equation*}
where $w_h$ are weight functions which converge to $1$ in suitable sense as $h$ tends to $0$.
We choose them to be piecewise constant
\begin{equation*}
 w_h(x) = \sum_{j=0}^{m-1} w_j \chi_{[(j-1)h, jh]}(x),
\end{equation*}
continuously extended to $x=1$, where $\chi_A$ is the characteristic function 
of the set $A$, that is $\chi_A(x)=1$ if 
$x\in A$ and $\chi_A(x)=0$ if $x \notin A$.
A function $u_h$ is a solution of  $(\mathrm{P}_{u_h}^\textrm{mod})$ 
if and only if  $(\mathrm{P}_{u_h}^\textrm{mod})$ is 
satisfied for $v_h=\varphi_i$ for all $i=0,\ldots, m-1$. 
Hence, we can rewrite $(\mathrm{P}_{u_h}^\textrm{mod})$ in the following form
\begin{equation*}
\begin{cases}\displaystyle
\sum_{k=0}^{m-1} \ddot u_k(t) \int_0^L \varphi_k(x)\varphi_i(x)w_h(x) \dd x \displaystyle +
\sum_{k=0}^{m-1}  u_k(t) \int_0^L \varphi_k'(x)\varphi_i'(x)\dd x\\ =
-\lambda_h(t)\varphi_i(0) + \displaystyle \int_0^L  f (x,t)w_h(x)\varphi_i(x)\dd x, \\
\lambda_h(t)\in\partial I_{[0,\infty)}(u_h(0,t)),\\
u_h(\cdot,0)= u_h^0 \quad\text{ and }\quad u_{h,t}(\cdot,0)=v_h^0,
 \end{cases}
\end{equation*}
for all $i=0,\cdots, m-1$. This is a problem of the type 
\begin{equation*}
(\mathrm{P}_{u_h}^\textrm{mod*})\hspace{1em}
\begin{cases}\displaystyle
\sum_{k=0}^{m-1} M_{ik} \ddot u_k(t) + \sum_{k=0}^{m-1} S_{ik} u_k(t) 
= f_i(t) - \lambda_h(t) \delta_{i0},\\
\lambda_h(t) \in \partial I_{[0,\infty)} (u_h(0,t)),\\
u_h(\cdot,0)= u_h^0 \quad\text{ and }\quad u_{h,t}(\cdot,0)=v_h^0,
 \end{cases}
\end{equation*}
where $\delta_{i0}$ is the Kronecker symbol, 
$f_i(t)=\int_0^L f(x,t) w_h(x)\varphi_i(x) \dd x$. 
The symmetric matrices $M=(M_{ik})$ and $S=(S_{ik})$ can
be computed directly from the formulas 
\begin{equation*}
 M_{ik} = \int_0^L \varphi_k(x)\varphi_i(x)w_h(x) \dd x
 \quad\text{ and }\quad
 S_{ik} = \int_0^L \varphi_k'(x)\varphi_i'(x) \dd x,
\end{equation*}
that is 
\begin{equation*}
\begin{aligned}
& M_{00}= \frac{h}{3}w_0, \quad
 M_{ii}=\frac{h}{3} (w_{i-1}{+}w_i )\text{ for } i=1, \ldots, m-1, \quad
 M_{i,i-1}=\frac{h}{6} w_i\text{ for } i=0, \ldots, m-2,\\&
 S_{00}=\frac1{h},\quad
 S_{ii}=\frac2{h}\text{ for } i=1, \ldots, m-1,\quad
 S_{i,i+1}=-\frac1h\text{ for } i=0, \ldots, m-2.
 \end{aligned}
\end{equation*}
In matrix representation, we have
\begin{equation*} 
M= \frac{h}{3}
\begin{pmatrix}
w_0      & \tfrac{1}{2}w_0       &~~0      &\cdots      &   \cdots    & 0  \\
 \tfrac{1}{2}w_0      & w_0+w_1      & \tfrac{1}{2}w_1      &  ~ \ddots  &       &~ \vdots   \\
~0      & \tfrac{1}{2}w_1       & w_1+w_2     & ~ ~\ddots  &  \ddots     & ~\vdots   \\
\vdots& ~\ddots  & ~\ddots & ~ \ddots   & ~\ddots      & ~0   \\
 \vdots        &          &   \ddots      &\tfrac{1}{2}w_{m-3}          & w_{m-3}+w_{m-2}    & \tfrac{1}{2}w_{m-2} \\
0        &\cdots   & \cdots        & 0        & \tfrac{1}{2}w_{m-2}     &w_{m-2}+w_{m-1}
 \end{pmatrix}
\end{equation*}
and 
\begin{equation*} 
S= \frac{1}{h}
\begin{pmatrix}
~~1      & -1       &~~0      &\cdots      &   \cdots    & 0  \\
-1       & ~~2      & -1      &  ~ \ddots  &       &~ \vdots   \\
~~0      & -1       & ~~2     & ~ ~\ddots  &  \ddots     & ~\vdots   \\
\vdots& ~\ddots  & ~\ddots & ~ \ddots   & ~\ddots      & ~0   \\
 \vdots        &          &   \ddots      & -1         & ~~2   & -1\\
0        &\cdots   & \cdots        & ~~0        & -1    &~ ~1
 \end{pmatrix}.
\end{equation*}
Note that \(M\) and \(S\) are usually called mass and stiffness matrices, respectively.
Consider first the problem $(\mathrm{P}_{u_h}^\textrm{mod*})$ for $i=0$. We have 
\begin{equation} \label{p_mod_i=1}
\begin{cases}
\displaystyle{\frac{h}{3} w_0 \Bigl(\ddot u_0 {+} \frac{1}{2} \ddot u_1\Bigr)+
  \frac{1}{h}( u_0{-} u_1) = f_0-\lambda_h(t),}\\
\lambda_h(t) \in \partial I_{[0,\infty)}(u_h(0,t)).
\end{cases}
\end{equation}
For $w_0 > 0$, this produces oscillations of $u_0$ which are not observed in the limit. 
To eliminate these unphysical oscillations which are purely 
due to the numerical method, we assume $w_0=0$, so that 
\eqref{p_mod_i=1} becomes (note that $f_0=0$ if $w_0=0$)
\begin{equation} 
\label{p_mod_positive_part}
 \frac{1}{h}( u_1{-}u_0)\in \partial I_{[0,\infty)}(u_0),
\end{equation}
or equivalently, 
\begin{equation} \label{fucik}
 u_0(t)=u_1^+(t),
\end{equation}
where \(u_1^+\) denotes the positive part of \(u_1\).
Then for $i=1$, we obtain from $(\mathrm{P}_{u_h}^\textrm{mod*})$ that 
\begin{equation*} 
\frac{h}{3} w_1(\ddot u_1{+} \frac{1}{2} \ddot u_2)+ \frac{1}{h}( 2u_1 {-}  u_0{-}u_2) = f_1\\
\end{equation*}
and taking \eqref{p_mod_positive_part} into account, this yields 
\begin{equation*} 
\frac{h}{3} w_1\Bigl(\ddot u_1 {+} \frac{1}{2} \ddot u_2\Bigr)+ 
\frac{1}{h}( 2u_1{-}u_2) = f_1+ \frac{1}{h} u_1^+.
\end{equation*}
We have thus eliminated the singularities and problem  
$(\mathrm{P}_{u_h}^\textrm{mod*})$ can be equivalently stated as 
\begin{equation*}
 (\mathrm{P}_{u_h}^\textrm{mod**})
 \hspace{2em}
 \begin{cases}
\displaystyle{ \sum_{k=1}^{m-1} h M_{ik}^*\ddot u_k + \sum_{k=1}^{m-1} 
 \frac{1}{h}S_{ik}^* u_k = f_i + \frac{1}{h} u_1^+ \delta_{1i}\text{ for all }
 i=1,\ldots, m-1,}\\
 u_h(\cdot,0)= u_h^0 \quad\text{ and }\quad u_{h,t}(\cdot,0)=v_h^0,
 \end{cases}
\end{equation*}
with a Lipschitz continuous nonlinearity on the right hand side, and with matrices
\begin{equation*}
M_{ik}^*= \frac{1}{h} M_{ik}\text{ for }i,k=1,\ldots, m-1
 \quad\text{ and }\quad
S_{ik}^*= h S_{ik} \text{ for } i,k=1,\ldots, m-1.
\end{equation*}
Note that $(\mathrm{P}_{u_h}^\textrm{mod**})$ 
is related to a more general problem of Fu{\v{c}}\'{i}k spectrum (or `jumping nonlinearity'
in the old terminology); the reader is
referred to \cite{Fuc76BPJN} for further details.
Furthermore, we may observe that 
\((\mathrm{P}_{u_h}^\textrm{mod**})\) can be rewritten as follows:
\begin{equation*}
 (\mathrm{P}_{U_h}^\textrm{mod})
 \hspace{2em}
 \begin{cases}
\text{Find } {U}_h: [0,T] \rightarrow \mathbb{R}^{m-1}\text{ such that}\notag\\
\displaystyle{h M^*\ddot U_h + \frac{1}{h} S^* U_h = F+\frac{1}{h} u_1^+ e_1},\\
 U_h^0=U^0 \quad\text{ and }\quad \dot{U}_h^0=V^0,
 \end{cases}
\end{equation*}
where \(e_1\eqldef(1,0,\ldots,0)^{\tra}\), \(U_h\eqldef(u_1,\ldots, u_{m-1})^{\tra}\), 
\(F\eqldef(f_1,\ldots,f_{m-1})^{\tra}\), \(U^0\) and \(V^0\) approximate
the initial position and velocity.
Finally, the discrete energy associated with problem 
\((\mathrm{P}_{U_h}^\textrm{mod})\) is given by
\begin{equation}
\label{energy}
 \mathcal{E}_h (t) 
 \eqldef \Bigl(\frac{h}{2} {\dot U_h}^\tra {M^*} {\dot U_h} 
 + \frac{1}{2h} { U_h}^\tra {S^*} {U_h}  
 - \frac{1}{2h} (u_1^+)^2 - {U_h}^\tra F\Bigr)(t).  
\end{equation}

We assume that the weights $w_i$ are chosen in such a way that $M^*$ is invertible and
the matrix norm $\arrowvert\arrowvert\arrowvert (M^*)^{-1} \arrowvert\arrowvert\arrowvert$
of its inverse $(M^*)^{-1}$ is bounded above by a constant independent of $h$.
Below, we consider the following situations:
\begin{enumerate}[(\text{Mod} 1)]
\item $w_1 = \ldots = w_{m-1} \eqldef 1$\  (no redistribution);
\item $w_1 = \ldots = w_{m-1} \eqldef m/(m-1)$\ (uniform redistribution);
\item $w_1= 2$, $w_2 = \ldots = w_{m-1} \eqldef 1$\ (nearest neighbor redistribution).
\end{enumerate}
In these cases, the condition on $(M^*)^{-1}$ is satisfied.

Under this hypothesis, ($\mathrm{P}_{U_h}^{\text{mod}}$) can be rewritten as follows:
\begin{equation*}
(\mathrm{P}_{{\mathcal{U}}_h}^{\text{mod}})\hspace{1em}
\begin{cases}
\text{Find }{\mathcal{U}}_h: [0,T] \rightarrow \mathbb{R}^{m-1} 
\times\mathbb{R}^{m-1}\text{ such that}\\
{\mathcal{{U}}}_{h,t}= G({\mathcal{{U}}}_h),\\
{\mathcal{{U}}}_h(0)=(U^0,V^0)^{\tra},
\end{cases}
\end{equation*}
where we put \({\mathcal{{U}}_h}\eqldef ({U}_h, {V}_h)^{\tra}\) and
$G({\mathcal{U}}_h)\eqldef
(\frac{1}{h} (M^*)^{-1} V_h, - \frac{1}{h}S^*U_h
+ F +\frac1h u_1^+ e_1)^{\tra}$.
Observe that $G:\mathbb{R}^{m-1} \times\mathbb{R}^{m-1} 
\to \mathbb{R}^{m-1} \times\mathbb{R}^{m-1}$
is Lipschitz continuous. More specifically, for
${\mathcal{{U}}}_h^1, {\mathcal{{U}}}_h^2 \in \mathbb{R}^{m-1} \times\mathbb{R}^{m-1}$ we have
\begin{equation}\label{lipg}
 \|G({\mathcal{{U}}}_h^1)-G({\mathcal{{U}}}_h^2)\| 
\leq \frac{C}{h}\|\mathcal{{U}}_h^1-\mathcal{{U}}_h^2\|,
\end{equation}
with a constant $C$ independent of $h$, where $\| \cdot \|$
denotes the canonical norm in $\mathbb{R}^{m-1}\times\mathbb{R}^{m-1}$.
Existence and uniqueness results for the problem
\((\mathrm{P}_{{\mathcal{U}}_h}^{\text{mod}})\)
follow from the Lipschitz continuity of \(G({\mathcal{{U}}_h})\), for further details 
the reader is referred to~\cite{Crouzeix}. In particular, we have 
${\mathcal{{U}}}_h\in \C^1([0,T]; \mathbb{R}^{m-1} \times\mathbb{R}^{m-1})$.

\begin{lemma}
Let $N \in \mathbb{N}$ be given and let
$\Delta t = \frac{T}{N}$ be the time step. Then
the time discretization error for the Crank-Nicolson method to solve the semi-discrete problem 
$(\mathrm{P}_{\mathcal{U}_h}^{\text{\emph{mod}}})$ is of the order $\Delta t$. 
\end{lemma}

\bproof
Keeping $h$ fixed, we define discrete times $t_n = n\Delta t$ for $n=0, \dots, N$
and define the Crank-Nicolson discretization of Problem 
\((\mathrm{P}_{{\mathcal{U}}_h}^{\textrm{mod}})\)
by the recurrent formula
 \begin{equation} 
 \label{crank_V}
 \frac{\mathcal{U}_h^{n+1} -\mathcal{{U}}_h^n}{\Delta t}
 =  \frac{1}{2}
 (G(\mathcal{{U}}_h^{n+1}){+}G(\mathcal{{U}}_h^n)), \quad n=0, \dots, N-1,
\end{equation}
with initial condition $\mathcal{{U}}_h^0 = \mathcal{{U}}_h(0)$.\\[1mm]
We compare the exact solution ${\mathcal{{U}}}_h$ 
of \((\mathrm{P}_{{\mathcal{U}}_h}^{\textrm{mod}})\)
with the piecewise linear interpolation $\hat{\mathcal{U}}_h$ of the discrete sequence
$\mathcal{{U}}_h^n$, which is defined by the formula
 \begin{equation} 
 \label{inte}
 \hat{\mathcal{{U}}}_h(t) = \mathcal{{U}}_h^n + \frac{1}{\Delta t}(t - t_n) (\mathcal{{U}}_h^{n+1}
 -\mathcal{{U}}_h^n) \quad \text{for }\ t \in [t_n, t_{n+1}),
 \quad n=0, \dots, N-1,
\end{equation}
continuously extended to $t_N = T$. We have by \eqref{crank_V} that
 \begin{equation} 
 \label{inte2}
\hat{\mathcal{{U}}}_{h,t} - G(\hat{\mathcal{{U}}}_h) = D_n(t),
\end{equation}
where for $t \in (t_n, t_{n+1})$ we have
\begin{equation*}
\|D_n(t)\| = \left\| \frac{1}{2}
 (G(\mathcal{{U}}_h^{n+1}){+}G(\mathcal{{U}}_h^n)) - G(\hat{\mathcal{{U}}}_h)\right\|
 \le \frac{C}{h} \|\mathcal{U}_h^{n+1} -\mathcal{{U}}_h^n\|
 = \frac{C\Delta t}{h} \|\hat{\mathcal{{U}}}_{h,t}\|.
\end{equation*}
We cannot expect to obtain a higher order estimate, since $G$ is not continuously differentiable
because of the presence of the term $u_1^+$. On the other hand, ${\mathcal{{U}}}_h$
is of class $\textrm{C}^1$, and we may denote
\begin{equation*}
C_h = \max_{t\in[0,T]}\|{\mathcal{{U}}}_{h,t}(t)\|.
\end{equation*}
We thus have
 \begin{equation} 
 \label{inte3}
\|{\mathcal{{U}}}_{h,t} - \hat{\mathcal{{U}}}_{h,t}\|
\le \|G({\mathcal{{U}}}_h) - G(\hat{\mathcal{{U}}}_h)\|
+ \frac{C\Delta t}{h} \|{\mathcal{{U}}}_{h,t} - \hat{\mathcal{{U}}}_{h,t}\|  + \frac{C_h C\Delta t}{h}
\end{equation}
for a.\,e. $t \in (0,T)$. By virtue of \eqref{lipg} we obtain for $\Delta t < h/C$ that
 \begin{equation} 
 \label{inte4}
\|{\mathcal{{U}}}_{h,t} - \hat{\mathcal{{U}}}_{h,t}\|
\le \frac{C}{h - C\Delta t}\|{\mathcal{{U}}}_h - \hat{\mathcal{{U}}}_h\|
 + \frac{C_h C\Delta t}{h - C\Delta t}\quad \text{a.\,e. in }\ (0,T)\,,
\end{equation}
and the assertion follows from the Gronwall argument.
\eproof

The next goal is to prove the convergence of solutions to Problem 
\((\mathrm{P}_{u_h}^{\rm{mod}})\)
(in the form \((\mathrm{P}_{\mathcal{U}_h}^{\rm{mod}})\)) as $h \to 0$. We first
observe that \((\mathrm{P}_{u_h}^{\rm{mod}})\) is equivalent to
\begin{equation*}
(\mathrm{P}_{\textrm{var}}^{\rm{mod}})\hspace{1em}
\begin{cases}
\text{Find }u_h : [0,T] \rightarrow \calV_h\text{ such that for all } v_h\in \calK\cap\calV_h\\
\displaystyle{\int_{Q_{T}} (u_{h,tt}(x,t){-}f(x,t)) (v_h(x,t){-}u_h(x,t))w_h(x) \dd x\dd t}\\+ 
\displaystyle{\int_{Q_T}u_{h,x}(x,t)(v_{h,x}(x,t){-}u_{h,x}(x,t))\dd x\dd t \geq 0,}\\
u_h(x,0)= u_h^0(x)\quad \text{ and }\quad u_{h,t}(x,0)=v_h^0(x).
 \end{cases}
\end{equation*}
We assume that the initial data \(u_h^{0}\) and \(v_{h}^0\) satisfy
\begin{equation}
\label{eq:lim}
\lim_{h\rightarrow 0}\bigl(\norm[\calV]{u_h^{0} {-} u^0}+\norm[\calH]{v_h^{0} {-} v^0}\bigr)= 0. 
\end{equation}
The convergence of the solution \(u^h\) of 
\((\mathrm{P}_{\textrm{var}}^{\rm{mod}})\) to the solution of \eqref{wf_ch3} is proved below. 
To this aim, the same techniques detailed in the proof of Theorem 4.3 in~\cite{DPPR13}
are used.  Here, we allow for general weight functions including the above cases 
(Mod~1)--(Mod~3),
while in~\cite{DPPR13}, only the case (Mod~1) was considered.
The reader is also referred to~\cite{SchBer89NAWE}. 

Choosing in \((\mathrm{P}_{u_h}^{\rm{mod}})\) the test function 
$v_h(x) = u_{h,t}(x,t)$, we see that
for all \(\tau\in [0,T]\) the following energy relation  
\begin{equation}
 \label{ch3h}
\begin{aligned}
&\int_0^L(\abs{u_{h,t}(x,\tau)}^2 w_h(x){+}\abs{u_{h,x}(x,\tau)}^2)\dd x\\ 
&=\int_0^L(\abs{v^0_h(x)}^2 w_h(x){+}\abs{{u_{h,x}^{0}}(x)}^2)\dd x
+ 2\int_{Q_{\tau}}f(x,t)u_{h,t}(x,t)w_h(x)\dd x\dd t
\end{aligned}
\end{equation}
holds.

\begin{theorem}
\label{convergence mass_ch3}
Assume that the hypotheses of Theorem \ref{t1} and condition \eqref{eq:lim} hold.
Let there exist two constants $C_w > c_w > 0$ such that
$c_w \le w_h(x) \le C_w$ for all $h>0$ and a.\,e. $x \in (h,L)$, and let
\begin{equation*}
\lim_{h\to 0} \int_0^L|w_h(x) - 1|\dd x = 0\,. 
\end{equation*}
Then the solutions \(u_h\) of \((\mathrm{P}_{\rm{var}}^{\rm{mod}})\)
converge in the strong topology of
$\mathcal{V}$ to the unique solution $u$ of \eqref{wf_ch3} as \(h\) tends to \(0\).
\end{theorem}

\begin{proof}
The energy relation \eqref{ch3h} and the Gronwall lemma imply the existence of 
a constant $C>0$
independent of $h$ such that 
\begin{equation} \label{pt1}
\sup_{h > 0} \sup_{\tau\in[0,T]}\int_0^L(\abs{u_{h,t}(x,\tau)}^2 w_h(x) 
+ \abs{u_{h,x}(x,\tau)}^2)\dd x \le C\,.
\end{equation}
We have
\begin{equation} \label{pt2}
\int_0^h \abs{u_{h,t}(x,\tau)}^2\dd x = \int_0^h \abs{\dot u_0(\tau)\varphi_0(x)
+ \dot u_1(\tau)\varphi_1(x)}^2\dd x\,,
\end{equation}
hence, by virtue of \eqref{fucik},
\begin{equation} \label{pt3}
\int_0^h \abs{u_{h,t}(x,\tau)}^2\dd x \le C\int_0^L\abs{u_{h,t}(x,\tau)}^2 w_h(x)\dd x
\end{equation}
with a constant $C$ independent of $h$. We thus have
\begin{equation*}
\sup_{h > 0}  \sup_{\tau \in[0,T]}( \norm[\calV]{u_h(\cdot,\tau)}{+}
\norm[\calH]{u_{h,t}(\cdot,\tau)})\leq C\,. 
\end{equation*}
\\
Let us define 
 \(\mathcal{W}\eqldef\{u\in \L^\infty(0,T;\V): u_t\in \L^\infty(0,T;\H) \}\) 
endowed with the norm
\(\norm[\mathcal{W}]{u}\eqldef
{\textrm{ess}}\sup_{t\in[0,T]}
\bigl(
\norm[\V]{u(\cdot,t)}+\norm[\H]{u_t(\cdot,t)}
\bigr)\).
We conclude that there exists $\bar u \in \mathcal{W}$ and a subsequence, 
still denoted by $u_h$, such that 
\begin{subequations}
\label{eq:conv1}
\begin{align}
u_h&\rightharpoonup \bar u \quad \text{ in }\quad\L^\infty(0,T;\calV) \quad \text{ weak-*}\,,\\
u_{h,t}& \rightharpoonup \bar u_t \quad \text{ in }\quad\L^\infty(0,T;\calH) \quad \text{weak-*}\,.
\end{align}  
\end{subequations}
 Then, we may deduce from \eqref{eq:conv1} that 
\begin{equation}
\label{eq:conv2}
 u_h\rightharpoonup \bar u \quad \text{ in } \quad\mathcal{W}\quad \text{ weak-* }\,.
\end{equation}
Notice that for all \(\alpha < \frac{1}{2}\), we have 
$\mathcal{W}\hookrightarrow \mathrm{C}^{0,\frac{1}{2}}(Q_T)
\hookrightarrow \hookrightarrow  \mathrm{C}^{0,\alpha}(Q_T) $ hold 
(see~\cite{SchBer89NAWE}), where  
$\hookrightarrow $ and $\hookrightarrow \hookrightarrow$ denote the continuous
and compact embeddings,  respectively. Finally we find
\begin{equation}
\label{eq:conv3}
 u_h \rightarrow \bar u\quad \text{ in }\quad \mathrm{C}^{0,\alpha}(Q_T) 
\end{equation}
for all  \(\alpha < \tfrac{1}{2}\). Furthermore, \(u_h\) 
and \(\bar u\) belong to \(\calK\).
Our aim  is to establish that $\bar u = u$, that is, the limit \(\bar u\) coincides with the solution
$u$ of \eqref{wf_ch3}. However 
the elements of \(\calK\) are not smooth enough in time, then they should be approximated
before being projected onto \(\calV_h\).  Indeed this projection violates 
the constraint at \(x =0\), and therefore, the elements 
of $\calK$ need another approximation in order to satisfy the constraint strictly. 

Assume that $v\in\calK$ is an admissible test function for \eqref{wf_ch3}, that is, $v=\bar u$
for $t \geq T - \zeta$. 
For $\eta\leq \zeta/4$ we define an auxiliary function
\begin{equation*}
 v_{\eta}(x,t)\eqldef
 \begin{cases}
\displaystyle{
\bar u(x,t) + \frac{1}{\eta}\int^{t+\eta}_t(v(x,s)-\bar u(x,s)})\dd s + k(\eta)(L-x) \psi(t)
&\text{ if } \quad t\leq T-\eta,\\
\bar u(x,t) &\text{ if } \quad t > T-\eta,
\end{cases}
\end{equation*}
where $\psi$ is a smooth and positive function with the property
\(\psi=1\) on \([0,T -\eta/2]\) and 
\(\psi=0\)  on \([T-\eta/4, T]\).  We precise now how 
the parameter \(k(\eta)\) to ensure that \(v_{\eta}\in\calK\cap \L^{\infty}(0,T;\calV)\)
holds. 
Since $\bar u\in\mathrm{C}^{0,\frac{1}{2}}(Q_T)$, it follows that
there exists  a constant $C^*>0$ such that 
\begin{equation*}
\begin{aligned}
 &\Bigl{|} \bar u(0,t) - \frac1{\eta}\int_t^{t+\eta} \bar u(0,s)\dd s \Bigr{|} \leq 
\frac1{\eta} \int_t^{t+\eta} \abs{\bar u(0,t) - \bar u(0,s)} \dd s \\ &\leq 
\frac{C^* \norm[\mathcal{W}]{\bar u}}{\eta}\int_0^{\eta}\sqrt{s} \dd s = 
\frac{2}{3}C^* \norm[\mathcal{W}]{\bar u}\sqrt{\eta}.
\end{aligned}
\end{equation*}
Then for all $t\leq T- \eta/2$, we obtain 
\begin{equation*}
 v_\eta (0,t)\geq \frac1{\eta}\int_t^{t+\eta}v(0,s)\dd s-\frac{2}{3}C^*
 \norm[\mathcal{W}]{\bar u}\sqrt{\eta} + k(\eta)L \psi(t).
\end{equation*}
The choice  $k(\eta)=\frac{5}{3L}C^* \norm[\mathcal{W}]{\bar u}\sqrt{\eta} $
ensures that  for all \(t\leq T- \tfrac{\eta}{2}\), we get
 \begin{equation}\label{pt4}
  v_\eta (0,t)\geq C^* \norm[\mathcal{W}]{\bar u}\sqrt{\eta}.
 \end{equation}
Let $D_h : \mathrm{V} \to \mathrm{V}_h$ be the piecewise linear 
interpolation mapping defined by the formula
\begin{equation}\label{pt5}
D_h(z)(x) = z((i-1)h) + \frac{1}{h} (x - (i-1)h)(z(ih) - z((i-1)h))
\end{equation}
for $z \in \mathrm{V}$ and $x \in [(i-1)h, ih)$, continuously extended to $x = mh = L$.
From the Mean Continuity Theorem it follows that
\begin{equation}\label{pt6}
\lim_{h \to 0} \|D_h(z) - z\|_{\mathrm{H}} = 0
\quad\text{and}\quad\lim_{h \to 0} \|D_h(z) - z\|_{\mathrm{V}} = 0\,.
\end{equation}
The next step consists in choosing an adequate test function. Let us define
\begin{equation}
\label{test_v_h_ch3}
v_h (\cdot,t)\eqldef u_h(\cdot,t) + D_h(v_\eta - \bar u)(\cdot,t)
\end{equation}
for all  \(t \in [0,T]\). We have $v_h (0,t) =u_h(0,t)+ v_\eta(0,t) - \bar u(0,t)$,
and from \eqref{eq:conv3}--\eqref{pt4} it follows that
\(v_h\in\calK\cap \calV_h\), for all \(t\) provided \(h\) is
small enough.

Introducing \eqref{test_v_h_ch3} into \((\mathrm{P}_{\mathrm{var}}^{\rm{mod}})\)
and integrating by parts, it comes that
\begin{equation}
\begin{aligned}\label{pt7}
& -\int^L_0 \bar u_{h,t}(x,0) D_h(v_\eta - \bar u)(x,0)w_h(x)\dd x 
-\int_{Q_T} \bar u_{h,t}(x,t) D_h(v_{\eta,t} - \bar u_t)(x,t)w_h(x)\dd x\dd t\\&
 + \int_{Q_T} \bar u_{h,x}(x,t)(D_h(v_\eta - \bar u)_x(x,t)\dd x \dd t \geq 
 \int_{Q_T}f(x,t)
D_h(v_\eta - \bar u)(x,t)w_h(x)\dd x \dd t.
\end{aligned}
\end{equation}
By \eqref{pt6}, we have for $h \to 0$ the strong convergences
\begin{subequations}
\label{eq:conv4}
\begin{align}
D_h(v_{\eta,t} - \bar u_t) &\rightarrow v_{\eta,t} - \bar u_t
  \quad \text{in}\quad \L^2(0,T;\H),\\
D_h(v_{\eta,t} - \bar u_t) &\rightarrow  v_{\eta} -  \bar u\quad \text{in}\quad\L^2(0,T;\V). 
\end{align}
\end{subequations}
We now use \eqref{eq:conv1} and \eqref{eq:conv4}
to pass to the limit as \(h\to 0\) in \eqref{pt7} and find that
\begin{equation}\label{pt8}
\begin{aligned}
& -\int^L_0 v^0(v_\eta - \bar u)(x,0)\dd x 
-\int_{Q_T} \bar u_t(x,t)(v_{\eta,t} - \bar u_t)(x,t)\dd x\dd t\\&
 + \int_{Q_T} \bar u_x(x,t)(v_{\eta,x} - \bar u_x)(x,t)\dd x \dd t \geq 
 \int_{Q_T}f(x,t)(v_\eta - \bar u)(x,t)\dd x \dd t\,.
\end{aligned}
\end{equation}
The passage to the limit in \eqref{pt8} as $\eta \to 0$ is easy, and we conclude that $\bar u = u$
is the desired solution of \eqref{wf_ch3}.

It remains to prove that \(u_h\) converge strongly in $\mathcal{V}$.
Passing to the limit as $h \to 0$ in \eqref{ch3h} we obtain for a.\,e. $\tau\in (0,T)$ that
\begin{equation}
\label{eq:energy3_ch3}
\begin{aligned}
&\lim_{h\rightarrow 0}
\int_0^L(\abs{u_{h,t}(x,\tau)}^2w_h(x){+}\abs{u_{h,x}(x,\tau)}^2)\dd x\\&=
\int_0^L(\abs{v^0(x)}^2{+}\abs{{u_x^{0}}(x)}^2)\dd x+\int_{Q_{\tau}}f(x,t)u_t(x,t)\dd x\dd t\\&
= \int_0^L(\abs{u_t(x,\tau)}^2{+}\abs{u_x(x,\tau)}^2)\dd x.
\end{aligned}
\end{equation}
Since \(u_h \rightharpoonup u\) weakly-* in $\mathcal{W}$ and $w_h \to 1$ strongly
in every $\mathrm L^p$ with $p < \infty$, we conclude that $u_{h,t}\sqrt{w_h}$
converge weakly to $u_t$ in $\mathrm{L}^2(0,T;\mathrm{H})$. Since the norms
of $(u_{h,t}\sqrt{w_h}, u_{h,x})$ in $\mathrm{L}^2(0,T;\mathrm{H})$ converge to the norm
of $(u_t, u_x)$ in $\mathrm{L}^2(0,T;\mathrm{H})$, we see that $u_h$ converge strongly to
$u$ in $\mathcal{V}$. The limit solution $u$ is unique, hence the whole system $\{u_h: h>0\}$
converges to $u$ as $h\to 0$, which completes the proof.
\end{proof}


\section{The wave equation with Signorini and Dirichlet boundary conditions}
\label{waveeq_ch3}  

We consider a bar of  length $L=1$ clamped at one end and compressed at $t=0$. 
The bar  elongates under the elasticity effect; as soon as it reaches 
a rigid obstacle at time \(t_1\) then it stays in contact during the time \(t_2-t_1\) 
and it takes off at time \(t_2\), see Figure~\ref{square_ch3}.
This problem can be described mathematically  
by \eqref{wave_ch3}--\eqref{bound_cond_ch3} with the density of external forces 
\(f(x,t)=0\). We first describe how an analytical piecewise affine and periodic solution
to our problem can be obtained by using the characteristics method, the reader 
is referred to~\cite{daba13} for a detailed explanation. 
Then approximate solutions for some time-space discretizations   
and for several mass redistributions are exhibited and their efficiency are discussed. 

\subsection{Analytical solution}
\label{expsol2_ch3}

The domains considered here are defined by
\((0,L)\times (t_i,t_{i+1})\), \(i=0, 1, 2\), corresponding to 
the phases before, during and  after the impact, respectively. Each of them are divided into
four regions as it is represented on Figure~\ref{square_ch3}. 
We choose below \(t_i=i\), \(i=0,\ldots, 3\).
\begin{figure}[!ht]
\centering
\includegraphics[scale=0.8]{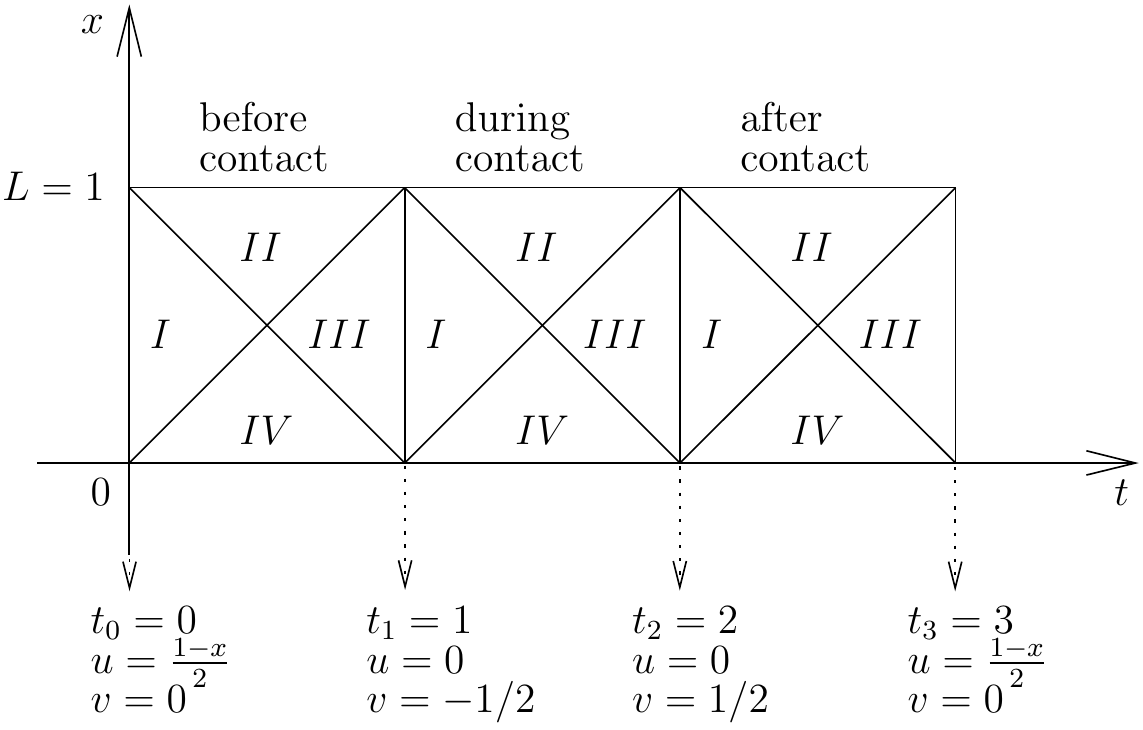}
\caption{The regions allowing to determine the value of u.}
\label{square_ch3}
\end{figure}

\noindent The domain $(0,1) \times (0,1)$ corresponding to the phase before the impact
is split into four regions according to the characteristics lines $x+t$ and $x-t$. Therefore
\begin{equation}
\label{eq:Ch3_2}
u(x,t)= 
\begin{cases}
{\frac{1{-}x}2} \text{ in the regions I, III}, \\
{\frac{1{-}t}2} \text{ in the regions II, IV}. 
 \end{cases}
 \end{equation}
 The domain $(1,2)\times (0,1)$ corresponding to the phase during the impact is also
 divided into four regions.  Then the solution \eqref{eq:Ch3_2} evaluated in the region IV allows us 
 to infer that  $u_t(\cdot,1)=-\frac{1}{2}$ and we conclude that
\begin{equation}
\label{eq:Ch3_4}
u(x,t)= 
 \begin{cases}
{\frac{t{-}t_1}2} \text{ in the region I}, 
{-\frac{x}{2}} \text{ in the region II},\\  
{\frac{x{-}1}2} \text{ in the region III}, \\
{\frac{t{-}t_1{-}1}2} \text{ in the region IV}.
 \end{cases}
 \end{equation}
The domain $(2,3)\times (0,1)$ corresponding to the phase after the impact is split 
into four regions. 
By using \eqref{eq:Ch3_4}, we get 
$u(\cdot,2)=0$ and $u_t(\cdot,2)=\frac{1}{2}$ which leads to
\begin{equation*}
u(x,t)=
\begin{cases} 
{\frac{t{-}t_2}2}  \text{ in the regions I, II}, \\
{\frac{1{-}x}2}  \text{ in the regions III, IV}.\\
 \end{cases}
 \end{equation*}
Since $u(\cdot,3)= u^0$ and $u_t(\cdot,3)= v^0$, the solution \(u(x,t)\) is periodic of period $3$.
Finally, note that \(\lambda=u_x(0,\cdot)\).

\subsection{Comparisons between different mass redistributions 
for some time-space discretizations}
 \label{numex2_ch3}

The time discretization is introduced in this section.  To this aim, we divide the time interval 
\([0,T]\) by \(n+1\) discrete time-points such that \(0=t_0<t_1<\ldots<t_n=T\). Let
\(U_h^n\), \(\dot{U}_{h}^n\), \(\ddot{U}_{h}^n\) and
\(\lambda^n\) be the approximations of the displacement \(U_h(t_n)\), the velocity 
\(\dot{U}_{h}(t_n)\), 
the acceleration \(\ddot{U}_{h}(t_n)\) and the Lagrange multiplier \(\lambda(t_n)\), respectively. 
We deal with some approximate solutions to Problem \eqref{wave_ch3}--\eqref{bound_cond_ch3} 
obtained by using several time-stepping methods like 
the Newmark, backward Euler and Paoli-Schatzman methods.
For each of these time-stepping methods,  approximate
solutions \((U_h^n,\lambda^n)\) are exhibited for several mass 
redistributions and they are compared to the analytical solution \((u, \lambda)\) 
introduced in Section \ref{expsol2_ch3}. In the numerical experiments presented below,
we distinguish the cases that the mass of the contact node that is not redistributed
(Mod~1), or uniformly redistributed on all the other nodes (Mod~2), or redistributed 
only on the nearest neighbor (Mod~3), according to the classification given in the previous section.
We show that the
efficiency of the mass redistribution method depends 
on the position of the nodes where the mass is redistributed. 
Indeed, the numerical experiments highlight that the closer from the contact node the 
mass is transferred better the approximate solutions are obtained. Then,
it is not surprising that the best approximate solution can be expected 
and indeed it is obtained
in the case (Mod~3), where all the mass of the contact node is transferred 
on the node preceding the contact node, see Figures~\ref{11}, \ref{21} and \ref{31}.
Note that the numerical simulations presented below were performed 
by employing the finite element library Getfem${++}$ (see \cite{REGET++}).  

\subsubsection{The Newmark methods}  
\label{Newmark}

The Taylor expansions of displacements and velocities neglecting terms 
of higher order are the underlying concept of the family of Newmark methods,
see \cite{NEM59}. These methods are 
unconditionally stable for linear elastodynamic problem for $\gamma \geq \tfrac{1}{2}$
and $\beta \geq \tfrac{1}{4}(\tfrac{1}{2}{+}\gamma)^2$, see \cite{Hug87, Krenk06}, but
they are also the most popular time-stepping schemes used to solve contact problems.
The discrete evolution for the contact problem 
\eqref{wave_ch3}--\eqref{bound_cond_ch3} is described
by the following finite difference equations:
\begin{equation}
\label{Newmark method3_ch3}
\begin{cases}
\text{find } U_h^{n+1} : [0,T] \rightarrow  \mathbb{R}^{m}\text{ and } 
\lambda^{n} : [0,T] \rightarrow \mathbb{R} \text{ such that:}\\
U_h^{n+1}= U_h^n +\Delta t \dot U_{h}^{n}+
\bigl(\frac{1}{2}{-}\beta\bigr)\Delta t^2 \ddot{U}_{h}^{n}+ 
\beta \Delta t^2\ddot{U}_{h}^{n+1},\\
\dot{U}_{h}^{n+1}=\dot{U}_{h}^n +(1{-}\gamma)\Delta t \ddot{U}_{h}^{n}
+\gamma \Delta t\ddot{U}_{h}^{n+1},\\
M \ddot{U}_{h}^{n+1}+S U_h^{n+1}=-\lambda^{n+1} e_0+F^{n+1},\\
0\leq u_0^{n+1}\perp \lambda^{n+1} \leq 0,
\end{cases} 
\end{equation}
where \(\Delta t\) is a given times step and \((\beta,\gamma)\) are
the algorithmic parameters, see \cite{Hug87,LauCCIM03}.
Note that \(U_h^0\), \(\dot{U}_{h}^0\) and
\(\lambda^0\) are given and \(\ddot{U}_{h}^0\) is evaluated by using
the third identity in \eqref{Newmark method3_ch3}.  
We are particularly interested in the case where \((\beta,\gamma)=(\frac14,\frac12)\). This method
is called the \emph{Crank-Nicolson method}, it is second-order
consistent and unconditionally stable in the unconstrained case. However the 
situation is quite different in the case of contact constraints, indeed the order of accuracy is
degraded; for further details, the reader is referred to \cite{Hug87,Krenk06, GRHA07}. 
The analytical solution \((u,\lambda)\) exhibited in Section
 \ref{expsol2_ch3} and the approximate solutions $(U^n_h , \lambda^n)$
 obtained for different mass redistributions are represented on
 Figure~\ref{11}.  The approximate solution obtained for the
  nearest neighbor redistribution (Mod~3 on Figure~\ref{11})
 gives much better accuracy than the mass redistribution on all the nodes preceding
 the contact node (see Mod~2 on Figure~\ref{11}) and no mass redistribution. This highlighted 
 that the choice for the mass redistribution plays a crucial role.

\begin{center}
 \begin{figure}[htb!]
   \begin{center}
      \subfigure{\includegraphics[height=5cm]
      {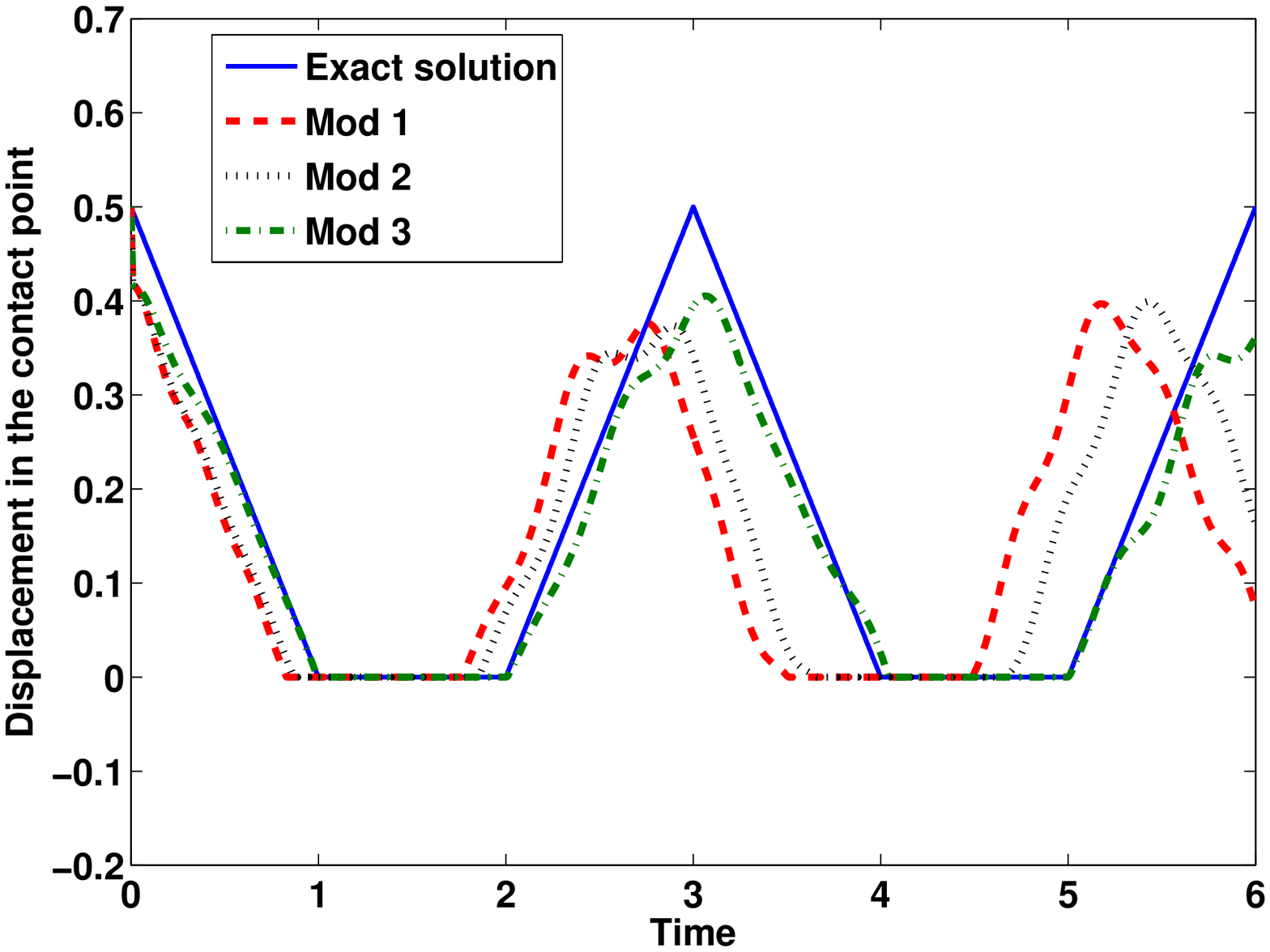}}\hspace*{-.3cm} 
       \subfigure{\includegraphics[height=5cm]
       {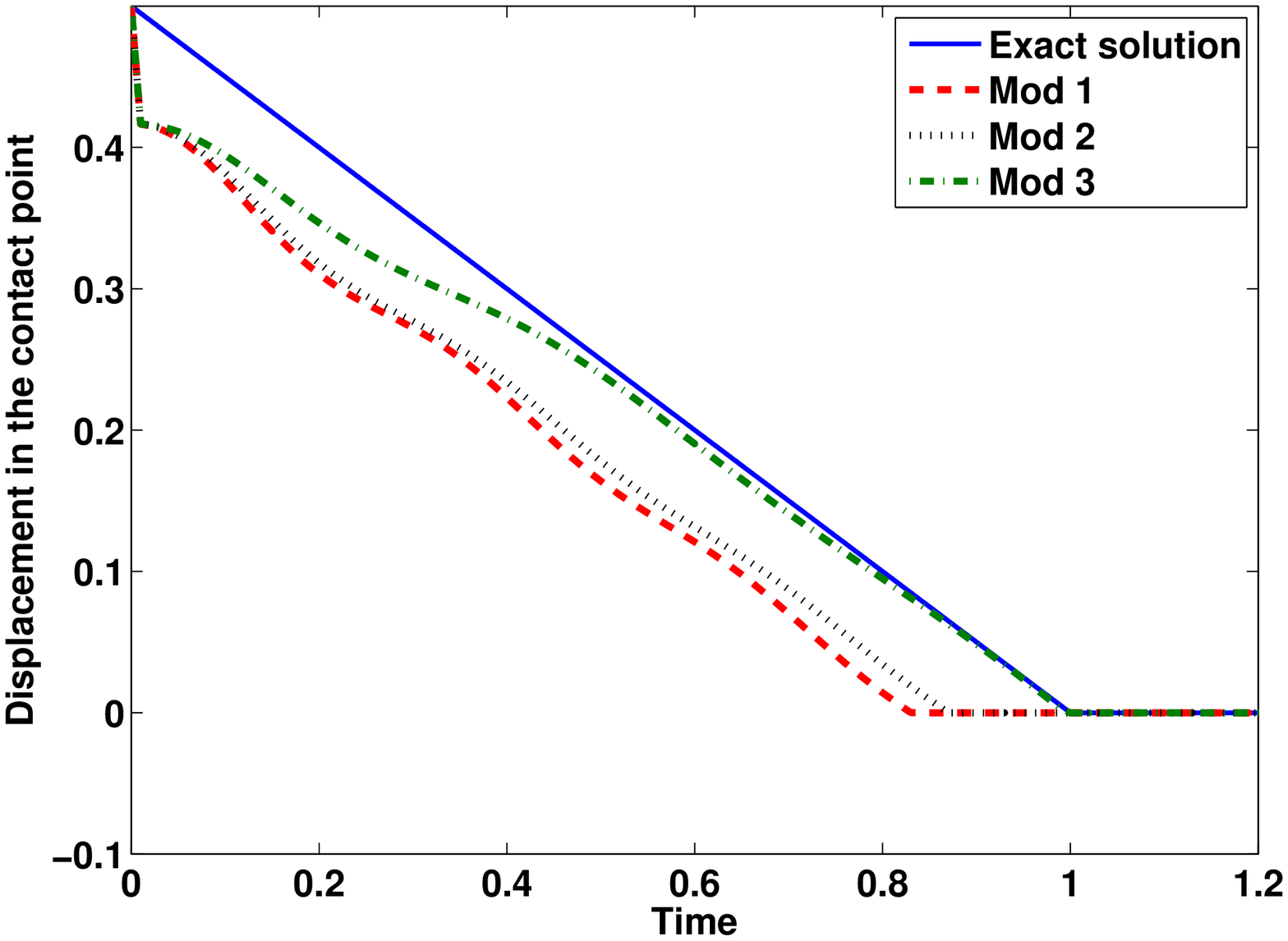}}
       \end{center}
        \vspace{-2em}
   \begin{center} 
      \subfigure{\includegraphics[height=5cm]
      {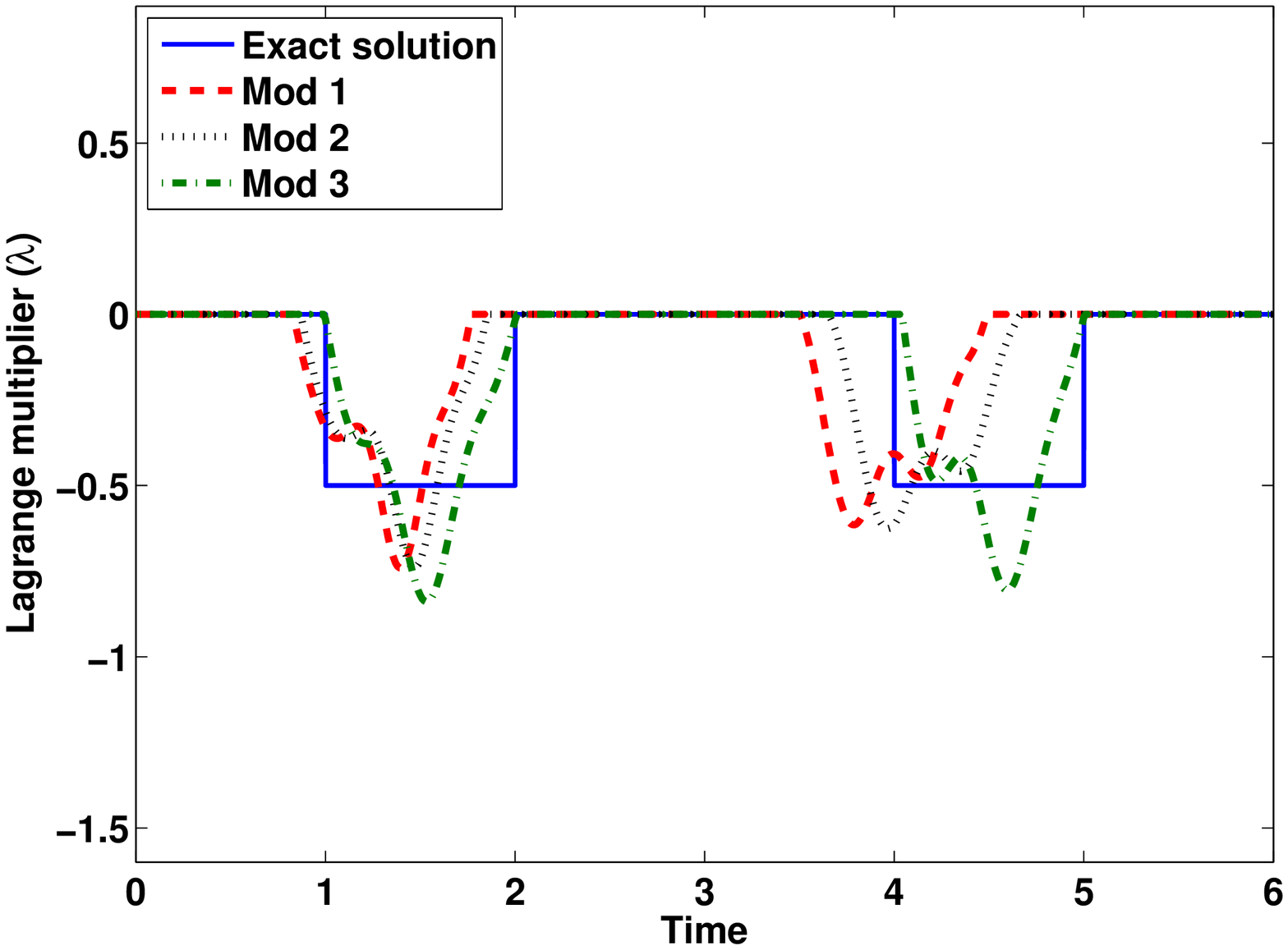}}\hspace*{-.3cm} 
       \subfigure{\includegraphics[height=5cm]
       {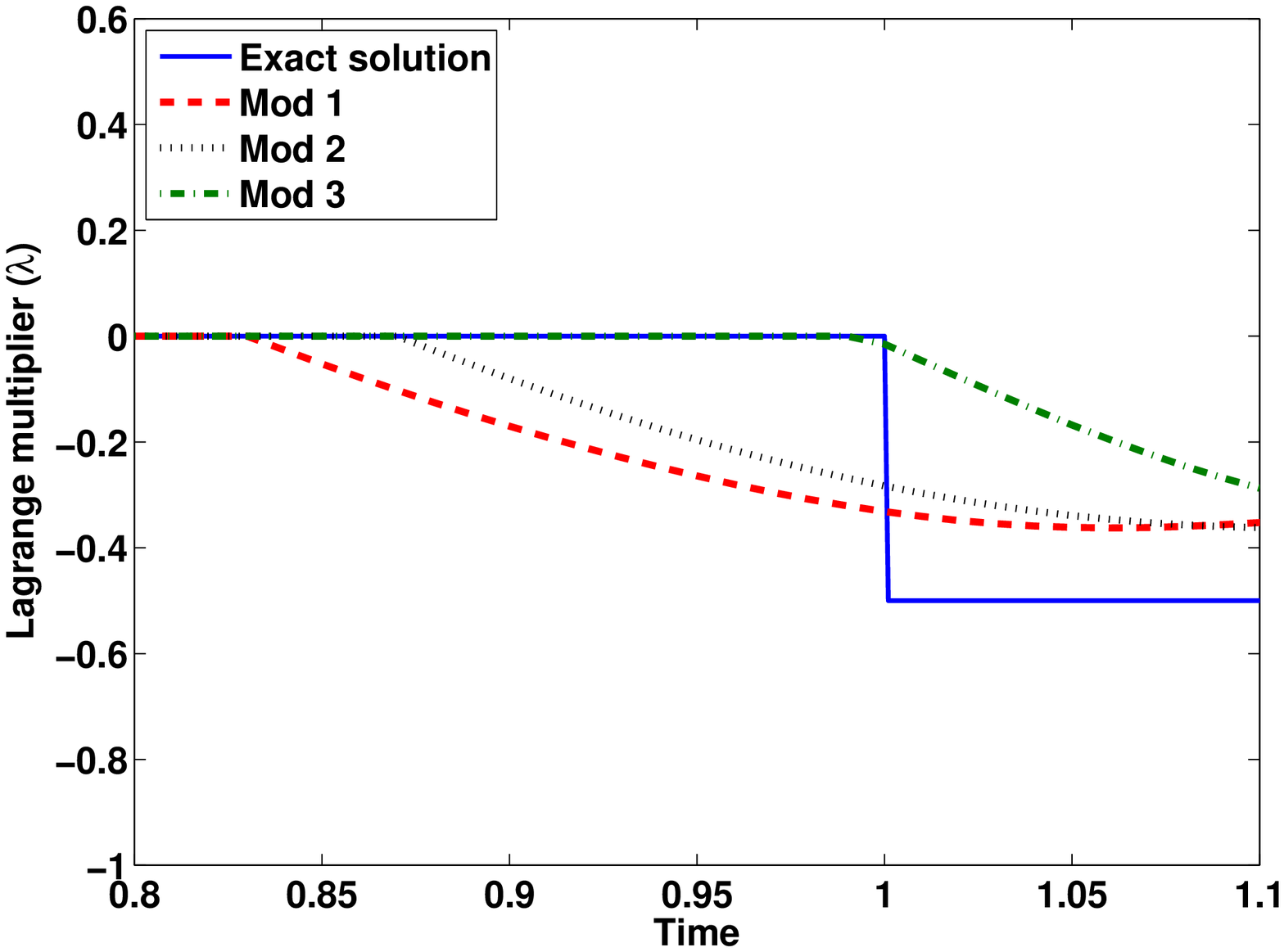}}
    \end{center}
    \vspace{-1.2em}
\caption{Comparison of analytical $(u,\lambda)$ and 
approximate $(U^n_h , \lambda^n)$ solutions 
for some modified mass matrices in the contact node with Crank-Nicolson method
($\Delta x = \frac{1}{6}$ and $\Delta t = \frac{1}{100}$).}
\label{11}
 \end{figure}
\end{center}

\begin{center}
 \begin{figure}[htb!]
   \begin{center}
      \subfigure{\includegraphics[height=5cm]{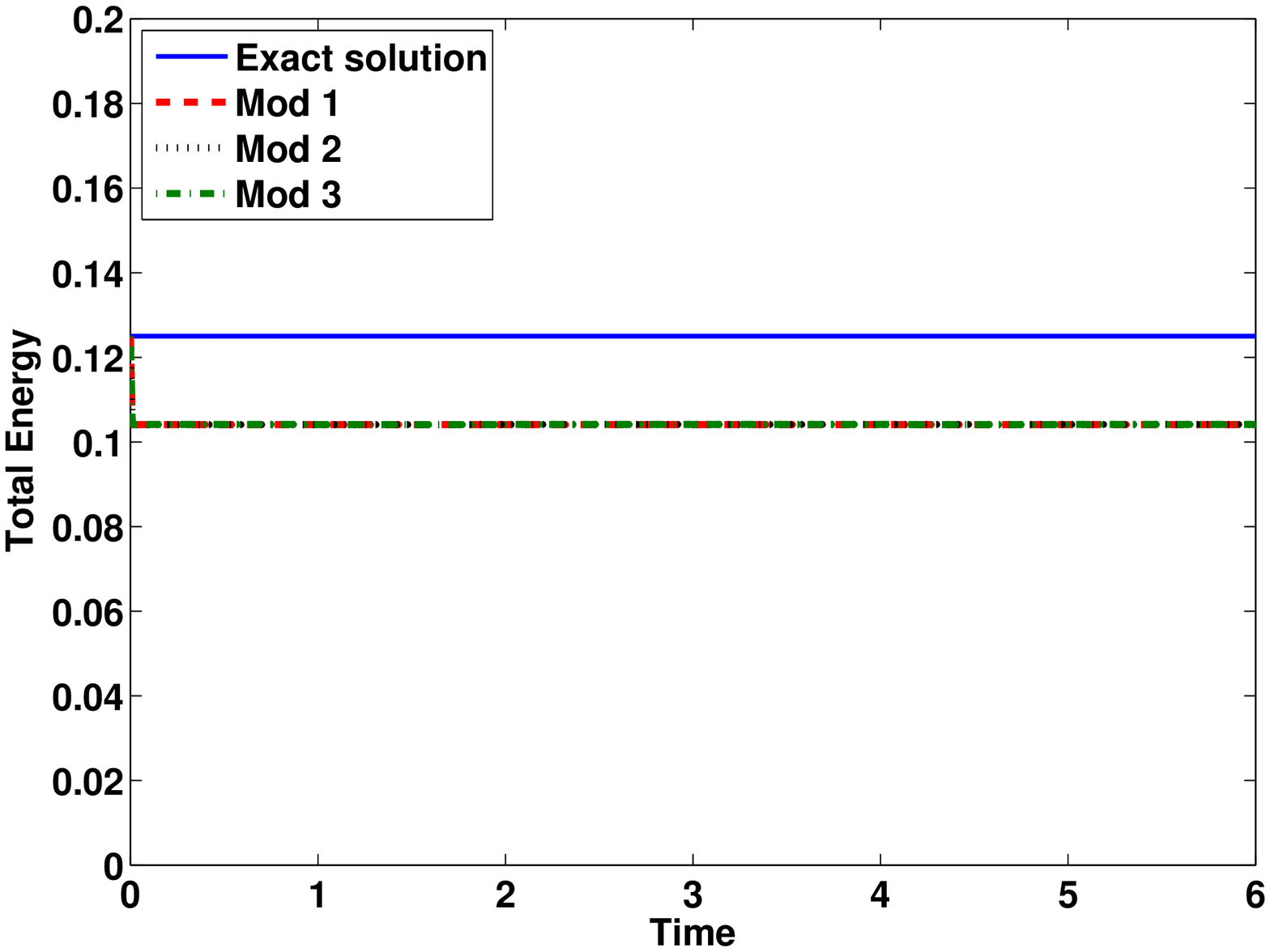}}
       \subfigure{\includegraphics[height=5cm]{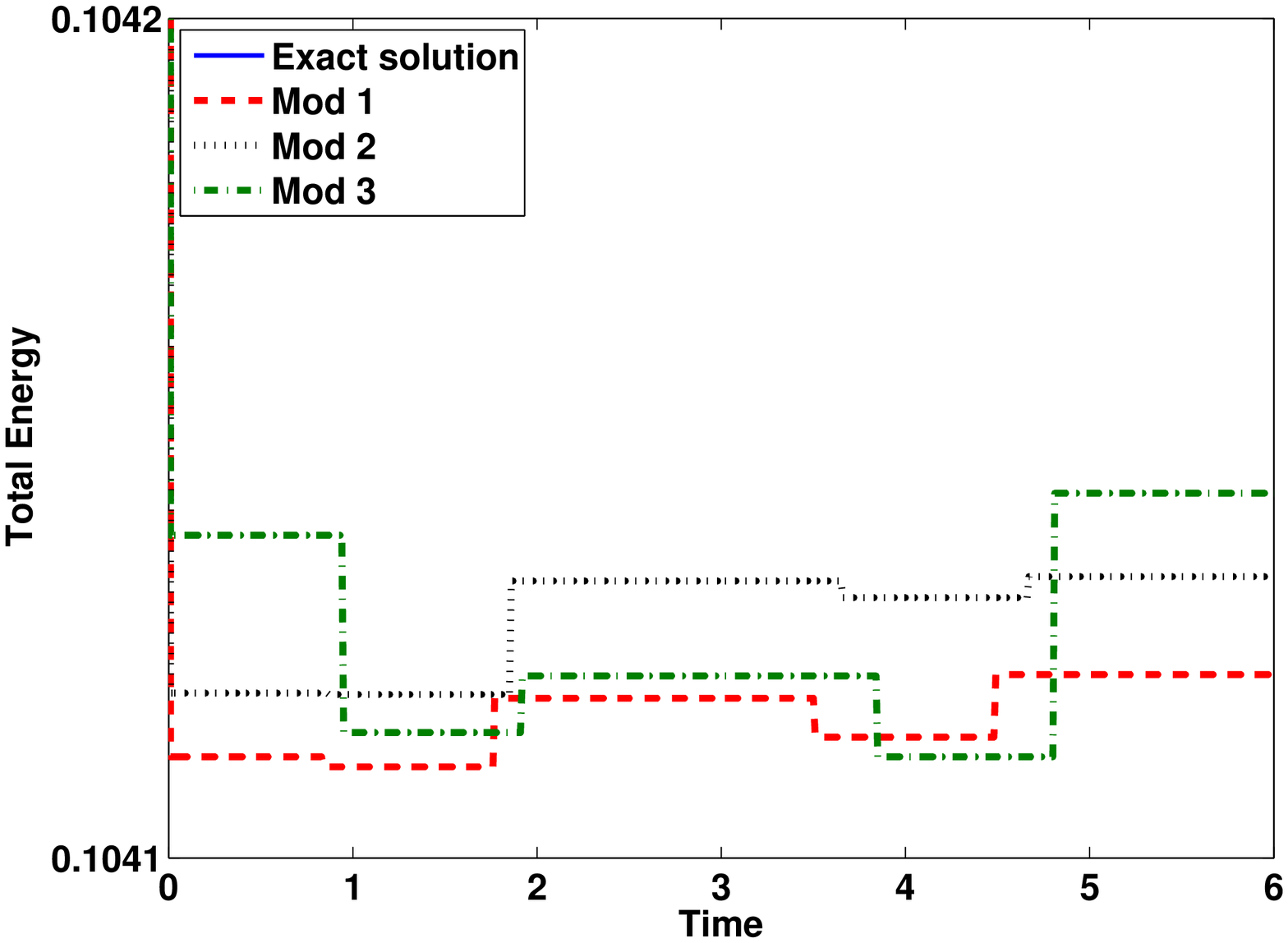}}
 \end{center}
 \vspace{-1.2em}
\caption{Comparison of the energy associated with analytical solution and
the energy associated with approximate solutions for some modified 
mass matrices with Crank-Nicolson method
($\Delta x = \frac{1}{6}$ and $\Delta t = \frac{1}{100}$). The figure on the right hand side
represents a zoom of the figure on the left hand side.}\label{12}
 \end{figure}
\end{center}

\newpage
\subsubsection{The backward Euler method}
\label{Euler_ch3}
We are concerned here with backward Euler's method which for the 
contact problem \eqref{wave_ch3}--\eqref{bound_cond_ch3} is described by the following 
finite difference equations:
\begin{equation}
\label{theta method3_ch3}
\begin{cases}
\text{find } U_h^{n+1} : [0,T] \rightarrow  \mathbb{R}^m\text{ and } 
\lambda^{n} : [0,T] \rightarrow \mathbb{R} \text{ such that:}\\
 U_h^{n+1}= U_h^n +\Delta t \dot{U}_{h}^{n+1},\\ 
 \dot{U}_{h}^{n+1}= \dot{U}_{h}^n +\Delta t \ddot{U}_{h}^{n+1},\\
M \ddot{U}_{h}^{n+1}+S U_h^{n+1}=-\lambda^{n+1} e_0+F^{n+1},\\
0\leq u_0^{n+1} \perp \lambda^{n+1} \leq 0.
\end{cases} 
\end{equation}
Note that \(U_h^0\),  \(\dot{U}_{h}^0\)
and \(\lambda^0\) are given and \(\ddot{U}_{h}^0\) is evaluated by using the third equality
in~\eqref{theta method3_ch3}.

 \begin{center}
 \begin{figure}[htp]
   \begin{center}
      \subfigure{\includegraphics[height=5cm]
      {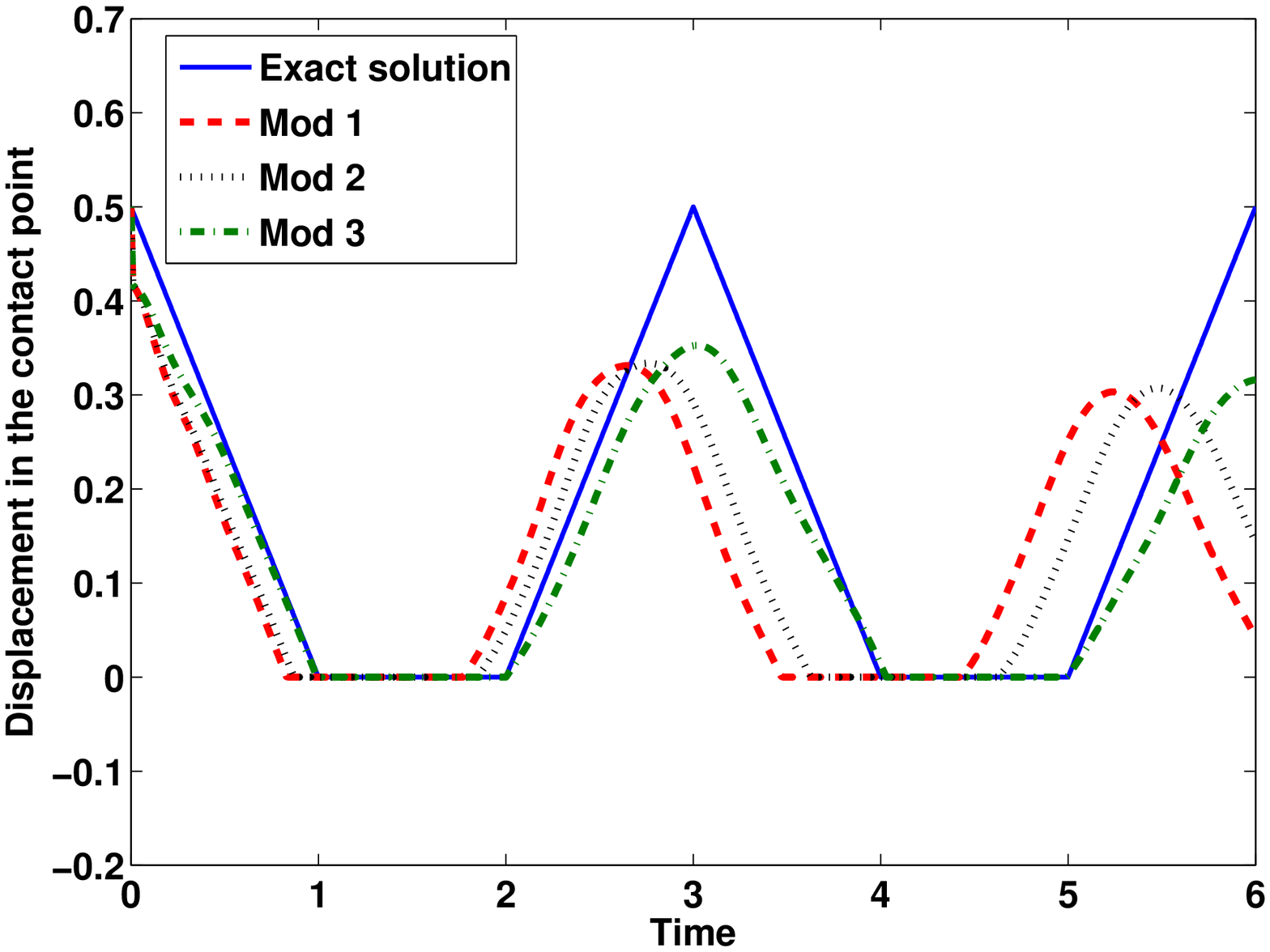}}\hspace*{-.3cm} 
       \subfigure{\includegraphics[height=5cm]{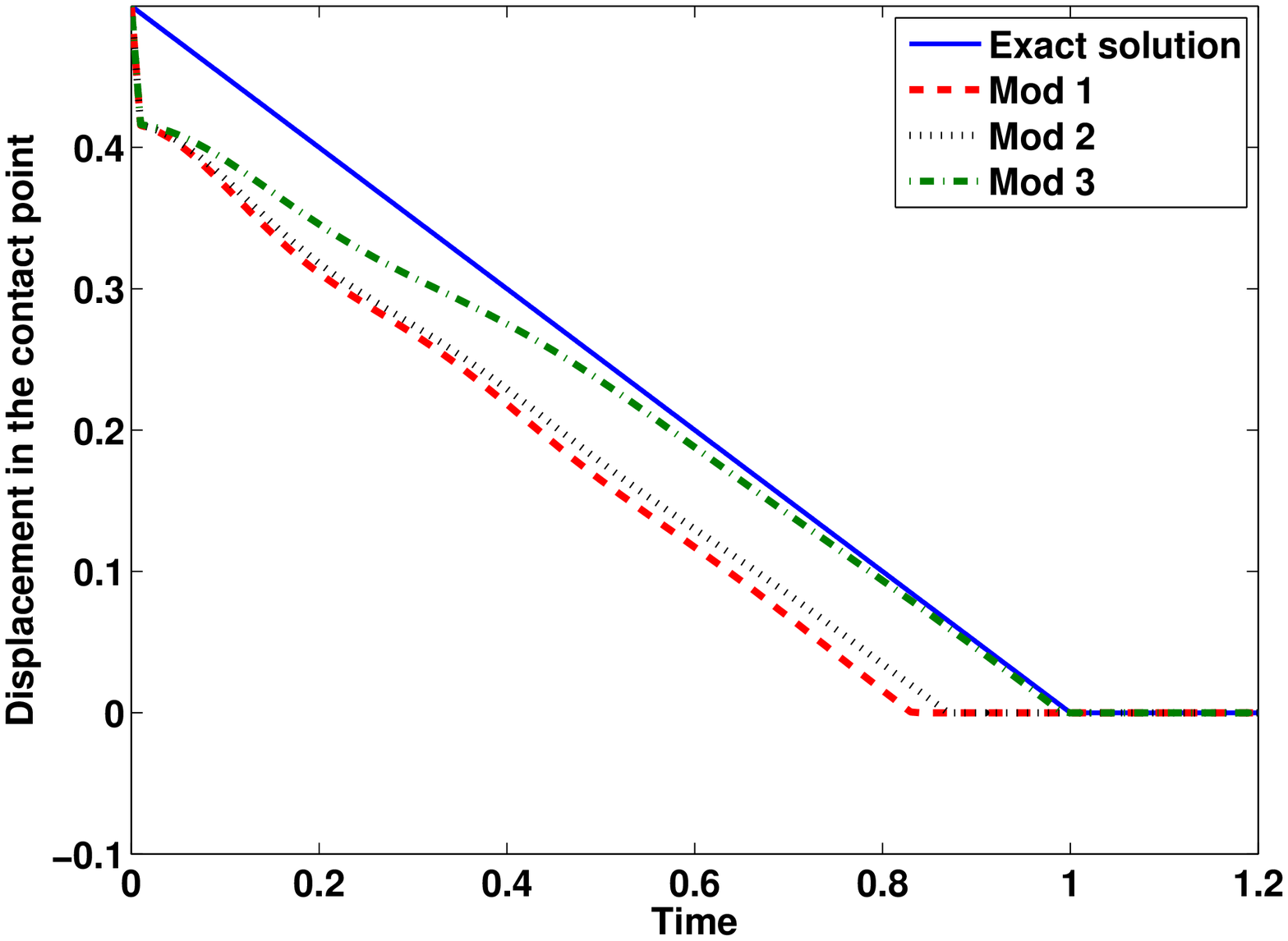}}
        \end{center}
        \vspace{-2em}
   \begin{center} 
    \subfigure{\includegraphics[height=5cm]
      {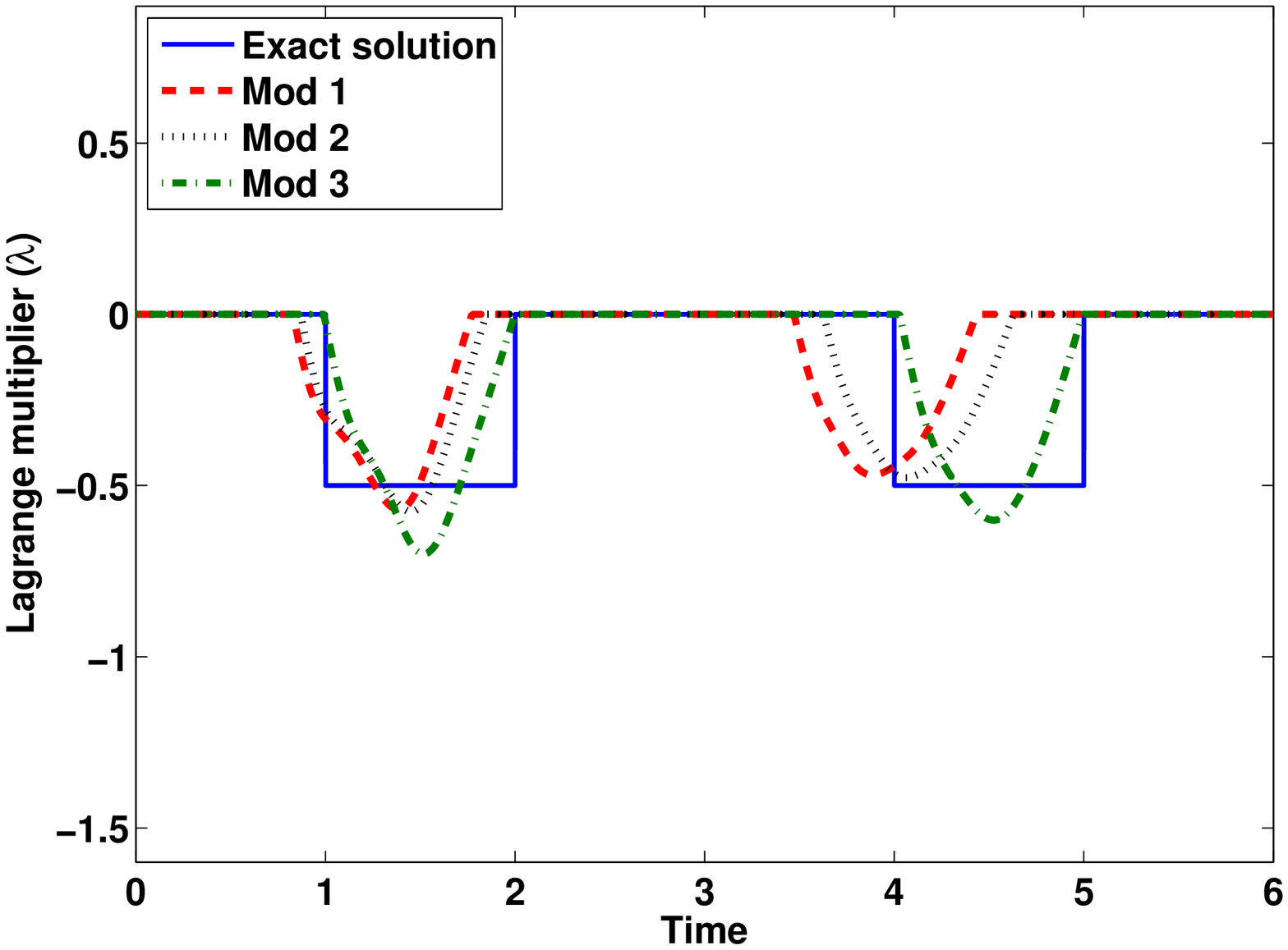}}\hspace*{-.3cm} 
       \subfigure{\includegraphics[height=5cm]{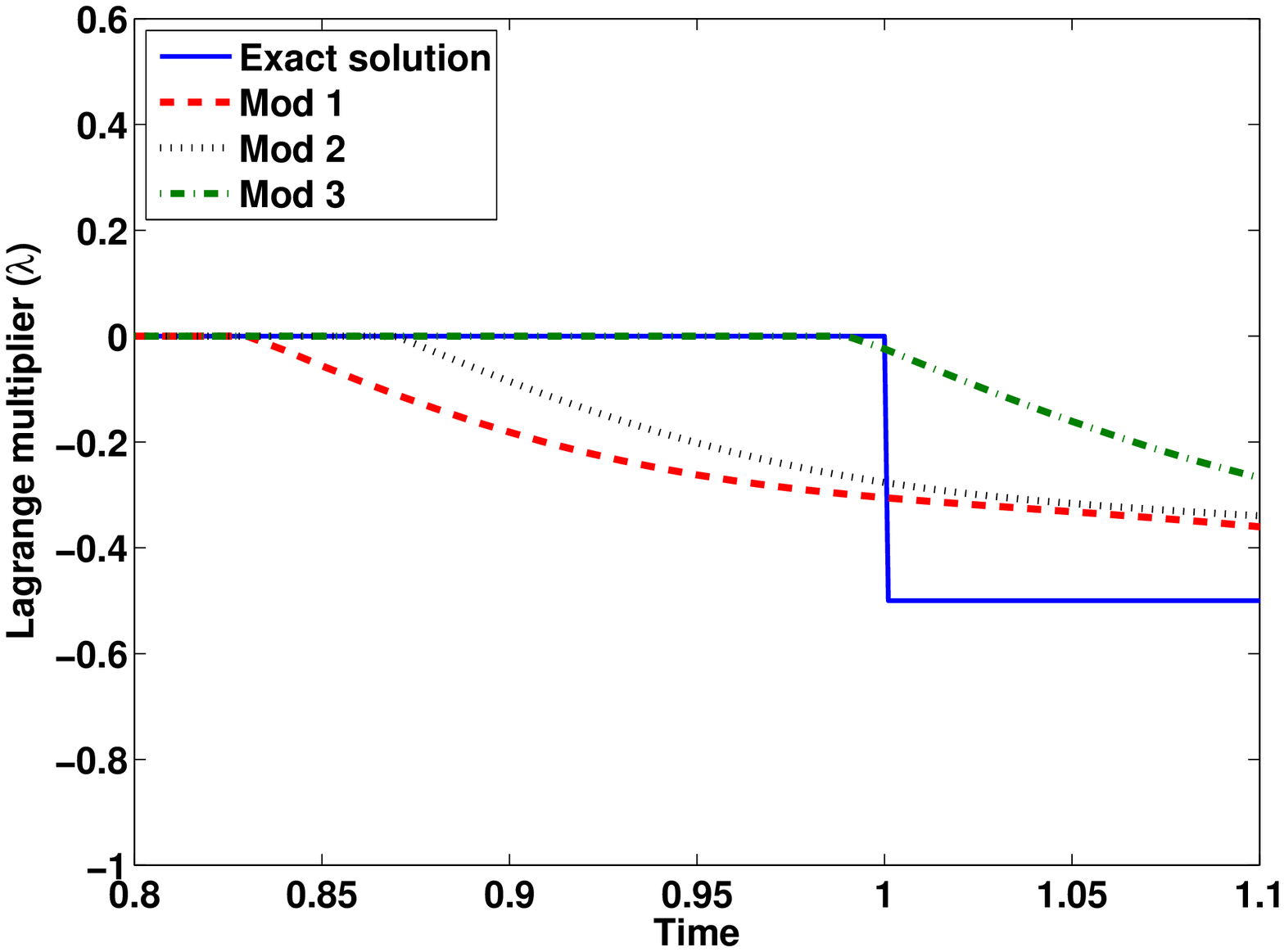}}
    \end{center}
   \vspace{-1.2em} 
\caption{Comparison of analytical $(u,\lambda)$ and 
approximate $(U^n_h , \lambda^n)$ solutions
for some modified mass matrices in the contact node with backward Euler method
($\Delta x = \frac{1}{6}$ and $\Delta t = \frac{1}{100}$).}\label{21}
 \end{figure}
\end{center}

\begin{center}
 \begin{figure}[htp!]
   \begin{center}
      \subfigure{\includegraphics[height=5cm]{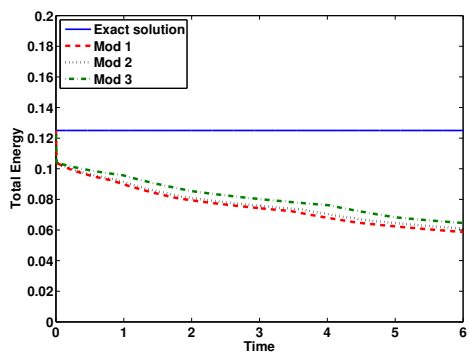}}
 \end{center}
\caption{Comparison of energy associated with analytical solution and
energy associated with approximate solutions for modified mass matrices with backward Euler method
($\Delta x = \frac{1}{6}$ and $\Delta t = \frac{1}{100}$).}\label{22}
 \end{figure}
\end{center} 

\subsubsection{The Paoli-Schatzman methods}
\label{Paoli_Schatz}

We focus on the so-called Paoli--Schatzman method
that consists to fix the contact constraint at an 
intermediate time step. Indeed the method proposed below is a slight
modification of Paoli-Schatzman method (see \cite{Paoli2001, PaoliSchat02})
which takes into account the kernel of the modified mass matrix. 
A simple application of  Paoli-Schatzman method based on Newmark scheme
to our problem with \(\gamma=\frac12\) leads to 
 \begin{equation}
 \label{paoli_method1_ch3}
\begin{cases}
\text{find } U_h^{n+1} : [0,T] \rightarrow  \mathbb{R}^m\text{ and } 
\lambda^{n} : [0,T] \rightarrow \mathbb{R} \text{ such that:}\\
\displaystyle{
\frac{M(U_h^{n+1}{-}2U_h^n{+}U_h^{n-1})}{\Delta t^2}+ S(\beta U_h^{n+1}{+}(1{-}2\beta)
U_h^n{+} \beta U_h^{n-1})=-\lambda^{n} e_0 \textrm{ for all } n\geq 2,}\\
\displaystyle{0\leq u_0^{n,e}=}\frac{u_0^{n+1}+eu_0^{n-1}}{1+e}\perp\lambda^{n} \leq 0,\\
U_0\text{ and }U_1\text{ given}.
 \end{cases}
\end{equation}
Here $e$ belongs to $[0,1]$ and is aimed to be interpreted as a
restitution coefficient.  Note that \(U_h^0\) and \(U_h^1\) are given
data and \(U_h^1\) can be evaluated by a one step scheme. We may
observe that taking \(M=M^ {\text{mod}}\) in \eqref{paoli_method1_ch3}, we
are not able to resolve the problem on the kernel of \(M^ {\text{mod}}\). That is the reason
why, \(SU_h^{n-1}\) as well as \(SU_h^n\) are projected on the
orthogonal of the kernel of \(M\).  
 We are interested  here in  Paoli-Schatzman's method with \((\beta,e)=(\frac14,1)\).
 Note that the stability result  immediately follows from \cite{DuP06VBCD}.
 
 \begin{center}
 \begin{figure}[htp!]
   \begin{center}
      \subfigure{\includegraphics[height=5cm]{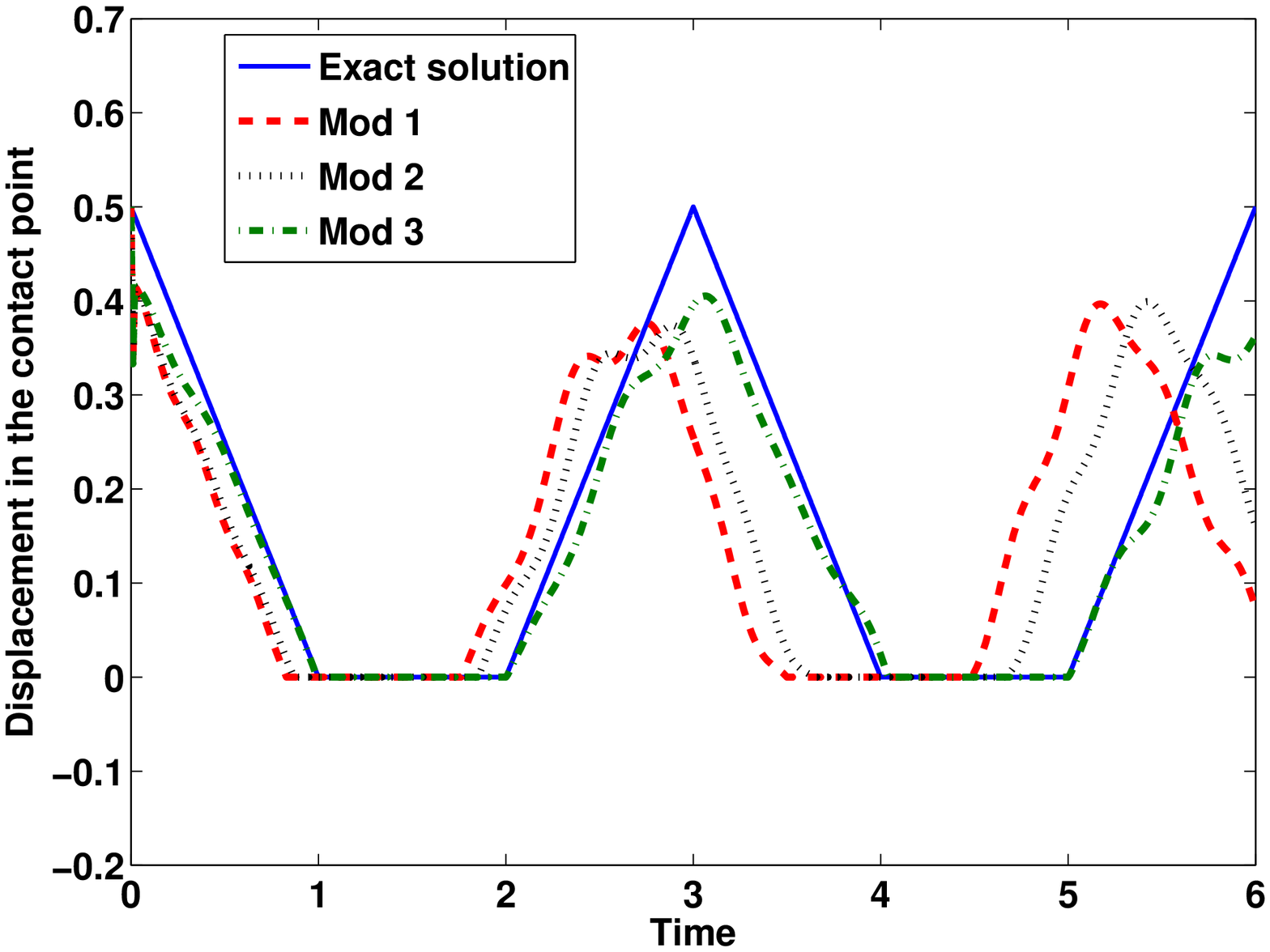}}\hspace*{-.3cm} 
       \subfigure{\includegraphics[height=5cm]{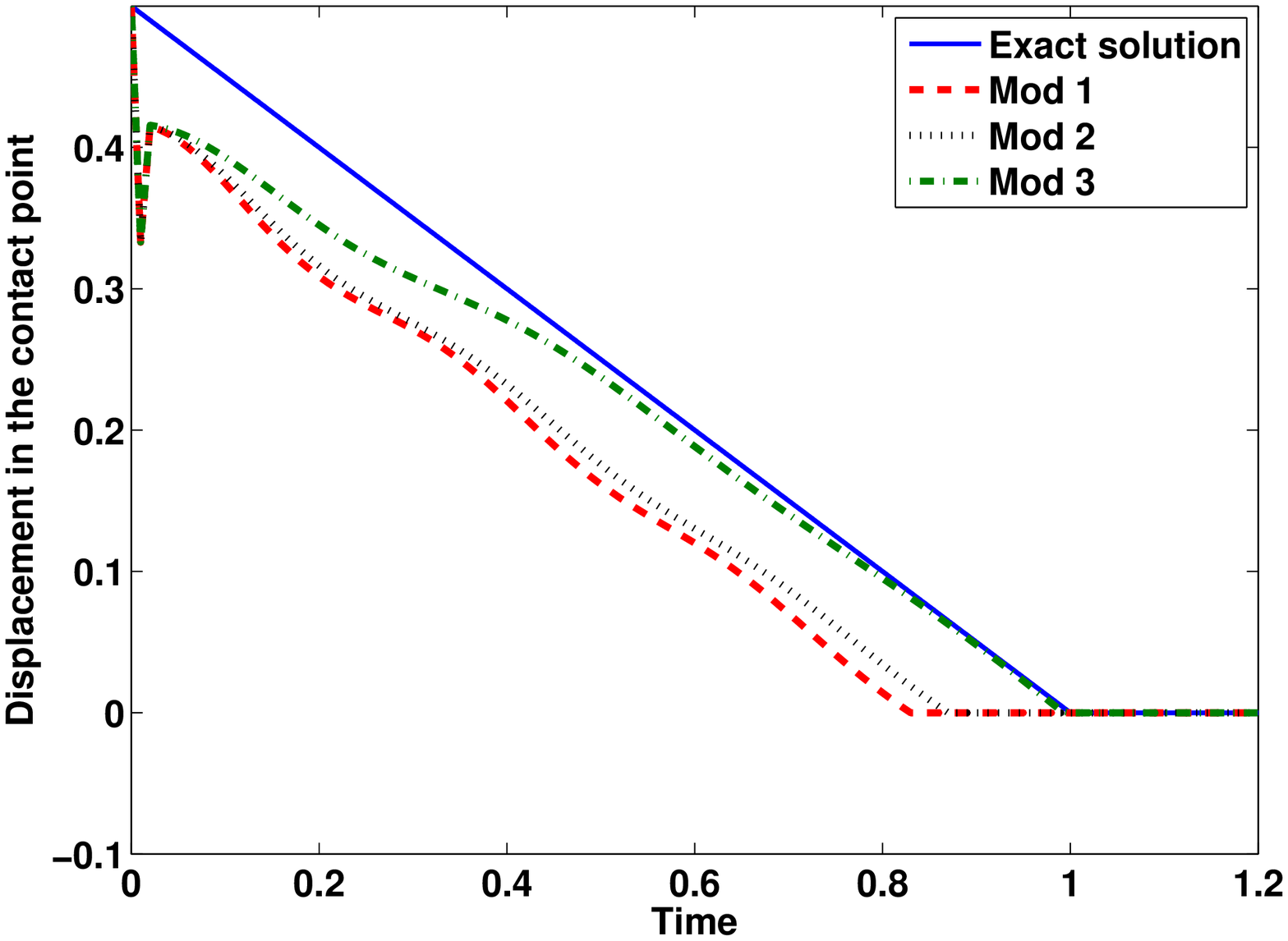}}
        \end{center}
       \vspace{-2em}
   \begin{center}
      \subfigure{\includegraphics[height=5cm]{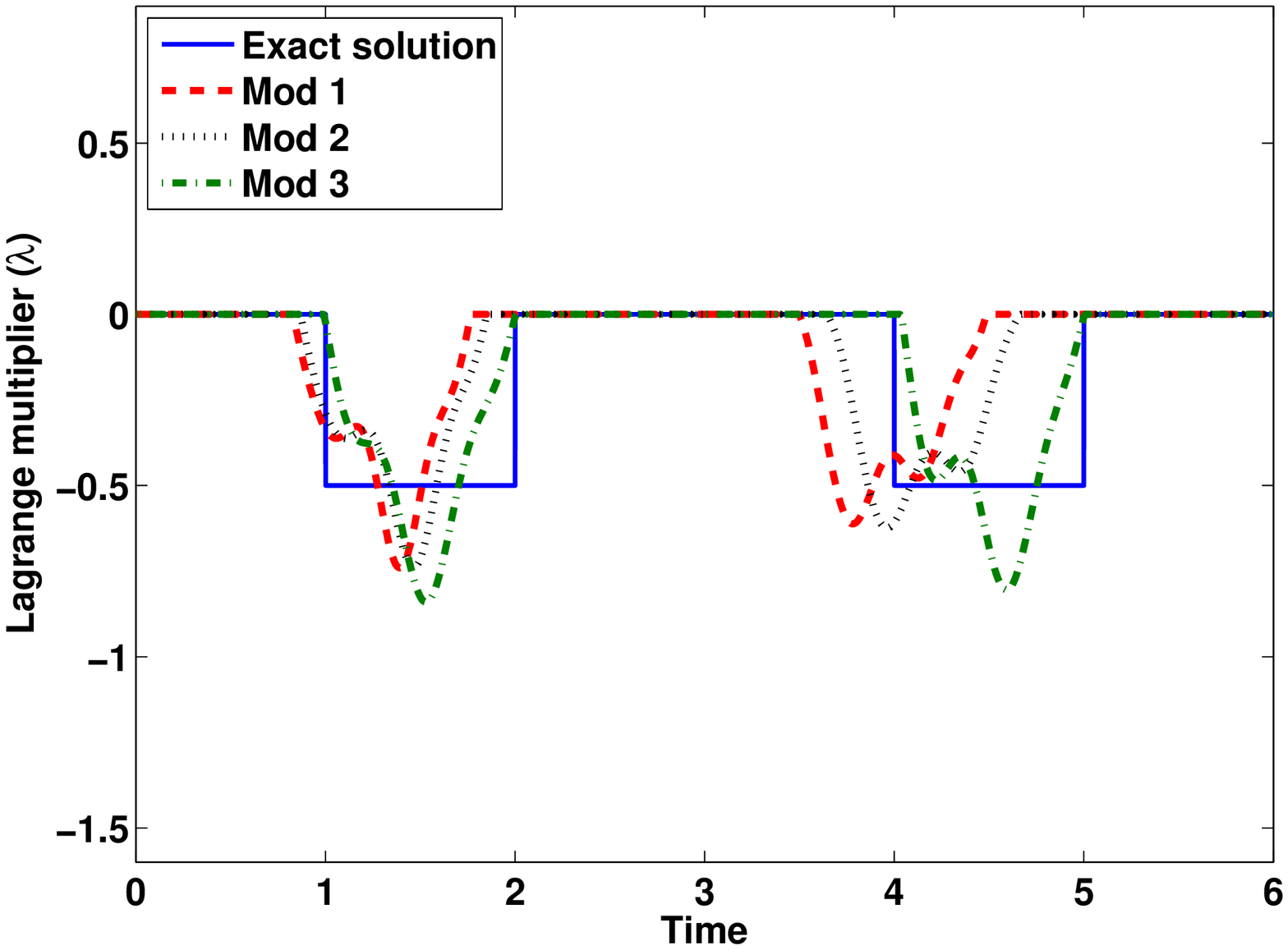}}\hspace*{-.3cm} 
       \subfigure{\includegraphics[height=5cm]{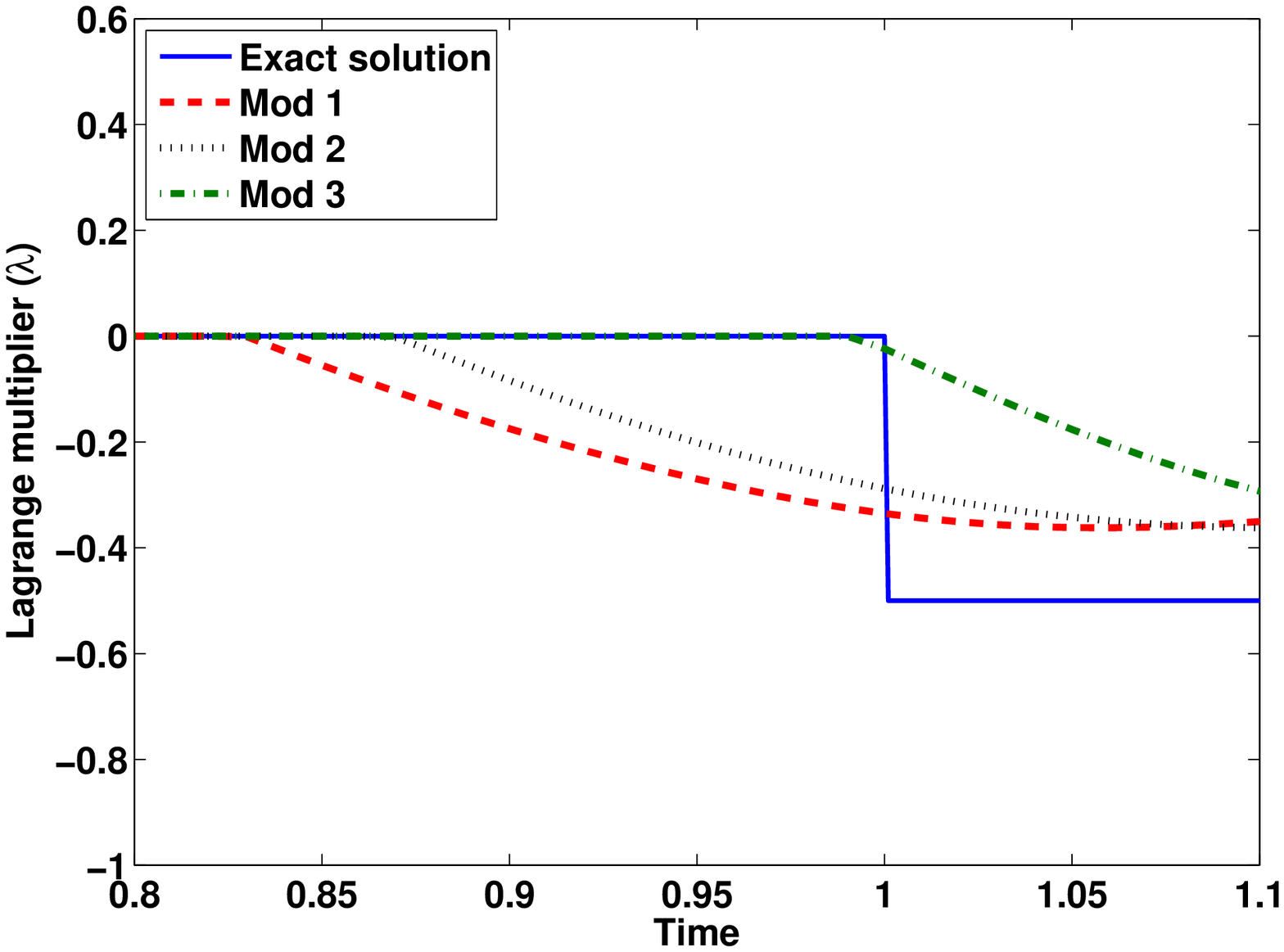}}
    \end{center}
\caption{Comparison of analytical $(u,\lambda)$ and
approximate $(U^n_h , \lambda^n)$ solutions
for some modified mass matrices in the contact node with  Paoli-Schatzman method
($\Delta x = \frac{1}{6}$ and $\Delta t = \frac{1}{100}$).}\label{31}
 \end{figure}
\end{center}

\begin{center}
 \begin{figure}[htp!]
   \begin{center}
      \subfigure{\includegraphics[height=5cm]{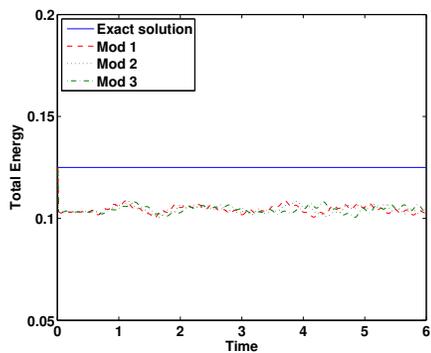}}
 \end{center}
\caption{Comparison of energy associated with analytical solution and
energy associated with approximate solutions for some modified mass matrices with  
Paoli-Schatzman method
($\Delta x = \frac{1}{6}$ and $\Delta t = \frac{1}{100}$).}\label{32}
 \end{figure}
\end{center}

\newpage
\section{A hybrid time integration scheme}
\label{hybrid_scheme}

Generally, the second order schemes illustrate one of 
the difficulties when solving contact problems, namely some oscillations for
energy associated with approximate solutions obtained for different choices of mass 
redistribution can be observed after each impact takes place, 
for instance see Figure~\ref{12}. 
To overcome this
problem, a hybrid time integration scheme is introduced in this section. More precisely,
the scheme $(\mathrm{P}_{U_h}^\textrm{mod})$ is modified to be an unconditionally stable 
and a second order in time scheme; the linear part of  $(\mathrm{P}_{U_h}^\textrm{mod})$
is discretized by using the midpoint method while the non-linear part is discretized by using 
the Crank-Nicolson as well as the midpoint methods. Observe that 
the midpoint method for the linear problem is energy conserving. 
The proposed hybrid time integration scheme inspired from~\cite{CHHIRE1-14} 
reads as follows:
\begin{equation*} 
(\mathrm{P}_{U_h}^\textrm{hyb})\hspace{2em}
\begin{cases}
\text{Find } U_h: [0,T] \rightarrow \mathbb{R}^{m-1}\text{ such that}\\
{U}_h^{n+1}={U}_h^{n}+\frac{\Delta t}{2}\big({\dot U}_h^{n+1}+ {\dot U}_h^{n}\big) \\
{\dot U}_h^{n+1}={\dot U}_h^{n}+\frac{\Delta t}{2}\big({\ddot U}_h^{n+1}+ {\ddot U}_h^{n}\big) \\
{M^*}{\ddot U}_{h}^{n+\frac{1}{2}}+{S^*}{U}_h^{n+\frac{1}{2}} =  F+ 
\frac{H(-u_1^n)}{2h}(u_1^n + u_1^{n+1})^+{e}_1 + \frac{H(u_1^n)}{2h} ((u_1^n)^+
+ (u_1^{n+1})^+){e}_1,
\end{cases}
\end{equation*}
where $V^{n+\frac{1}{2}}\eqldef\frac{V^{n+1}+V^{n}}{2}$ and  
$H$ is defined by
\begin{equation*}
H(s)\eqldef
\begin{cases}
1 \text{ if }s>0, \\
\tfrac{1}{2}\text{ if } s=0, \\
0\text{ otherwise}. 
\end{cases}
\end{equation*}
We assume that the density of external forces $F$ does not depend on time.  
Observe that $\frac{1}{2h}((u_1^n)^+ +(u_1^{n+1})^+) {e}_1$ and 
$\frac{1}{2h}(u_1^n + u_1^{n+1})^+{e}_1$ correspond to the contribution of Crank-Nicolson
and midpoint methods, respectively. 

The discrete evolution of the total energy is preserved in the purely elastic case 
when the density of external forces vanishes, see \cite{LauCCIM03}.  
However, the situation quite different 
in the case of contact constraints, the order of accuracy is
degraded and , for further details see \cite{Hug87,Krenk06,GRHA07}.  
Let us define now the energy evolution by
\(\Delta\mathcal{E}_h^n\eqldef \mathcal{E}_h^{n+1}-\mathcal{E}_h^n\),
where \(\mathcal{E}_h^n\) is assumed to be given by an algorithmic
approximation of the energy \(\mathcal{E}_h(t_n)\) defined in \eqref{energy}.
We evaluate now \(\Delta\mathcal{E}_h^n\) by using the midpoint scheme.
More precisely, we get
\begin{equation*}
\Delta\mathcal{E}_h^n=({\dot U}_h^{n+\frac{1}{2} })^\tra {M^*} 
(\Delta t {\ddot U}_h^{n+\frac{1}{2} }) + 
({ U}_h^{n+\frac{1}{2} })^\tra {S^*}(\Delta t {\dot U}_h^{n+\frac{1}{2} }) -
(\Delta t {\dot U}_h^{n+\frac{1}{2} })^{\tra}F
+\tfrac{((u_1^{n})^+)^2-((u_1^{n+1})^+)^2}{2h}.
\end{equation*}
Since ${M^*}$ is symmetric matrix, it comes that
\begin{equation*}
 \Delta\mathcal{E}_h^n
 ={\Delta t}({\dot U}^{n+\tfrac{1}{2} })^\tra  
\bigl(\tfrac{H(-u_1^n)}{2h}(u_1^n + u_1^{n+1})^+{e}_1 + \tfrac{H(u_1^n)}{2h} ((u_1^n)^+
+ (u_1^{n+1})^+){e}_1\bigr)
+\tfrac{((u_1^{n})^+)^2-((u_1^{n+1})^+)^2}{2h},
\end{equation*}
which implies that
 \begin{equation*}
 \Delta\mathcal{E}_h^n= 
 (u_1^{n+1}{-}u_1^{n})  
\bigl(\tfrac{H(-u_1^n)}{2h}(u_1^n + u_1^{n+1})^+{e}_1 + \tfrac{H(u_1^n)}{2h} ((u_1^n)^+
+ (u_1^{n+1})^+){e}_1\bigr)
+\tfrac{((u_1^{n})^+)^2-((u_1^{n+1})^+)^2}{2h}.
 \end{equation*} 
 We establish below that the energy evolution by \(\Delta\mathcal{E}_h^n\) is nonpositive, 
 namely the energy associated with the hybrid scheme decreases in time. This result
 is summarized in the following lemma:
 
 \begin{lemma}
 \label{lemma_dissip}
Assume that the density of external forces $F$ does not depend on time.
Then the energy evolution \(\Delta\mathcal{E}_h^n\) is nonpositive for all $n>0$.     
\end{lemma}
\begin{proof}
We distinguish five cases depending on the values taken by 
$u^{n}_1$ and  $u^{n+1}_1$. More precisely, we get

\begin{enumerate}[\hspace{2em}1.]

\item If $u^{n}_1 < 0$ and  $u^{n+1}_1 \leq 0$ then
\begin{equation*}
 \Delta\mathcal{E}_h^n=  \tfrac{1}{2h} (u_1^{n+1} - u_1^{n})
 (u_1^n + u_1^{n+1})^+ -\tfrac{1}{2h}((u_1^{n+1})^+)^2 =0.
\end{equation*}

\item If $u^{n}_1 < 0$ and  $u^{n+1}_1 > 0$ then
\begin{equation*}
 \Delta\mathcal{E}_h^n=  \tfrac{1}{2h} (u_1^{n+1} - u_1^{n})
 (u_1^n + u_1^{n+1})^+ -\tfrac{1}{2h}((u_1^{n+1})^+)^2  < 0.
\end{equation*}

\item If $u^{n}_1 >0$ and  $u^{n+1}_1 \leq 0$ then
\begin{equation*}
 \Delta\mathcal{E}_h^n=  \tfrac{1}{2h} (u_1^{n+1} - u_1^{n})(u_1^{n})^+ 
 +\tfrac{1}{2h}((u_1^{n})^+)^2 < 0.
\end{equation*}

\item If $u^{n}_1 >0$ and  $u^{n+1}_1 > 0$ then
\begin{equation*}
 \Delta\mathcal{E}_h^n=  \tfrac{1}{2h} (u_1^{n+1} - u_1^{n})((u_1^{n+1})^+ +(u_1^{n})^+)
 -\tfrac{1}{2h}((u_1^{n+1})^+)^2 +\tfrac{1}{2h}((u_1^{n})^+)^2 =0.
\end{equation*}

\item If $u^{n}_1 = 0$ then 
\begin{equation*}
 \Delta\mathcal{E}_h^n=  \tfrac{1}{2h}(u_1^{n+1})(u_1^{n+1})^+
 -\tfrac{1}{2h}((u_1^{n+1})^+)^2 =0.
\end{equation*}

\end{enumerate}
This proves the lemma.
\end{proof}

  \begin{center}
 \begin{figure}[htp!]
   \begin{center}
      \subfigure{\includegraphics[height=5cm]{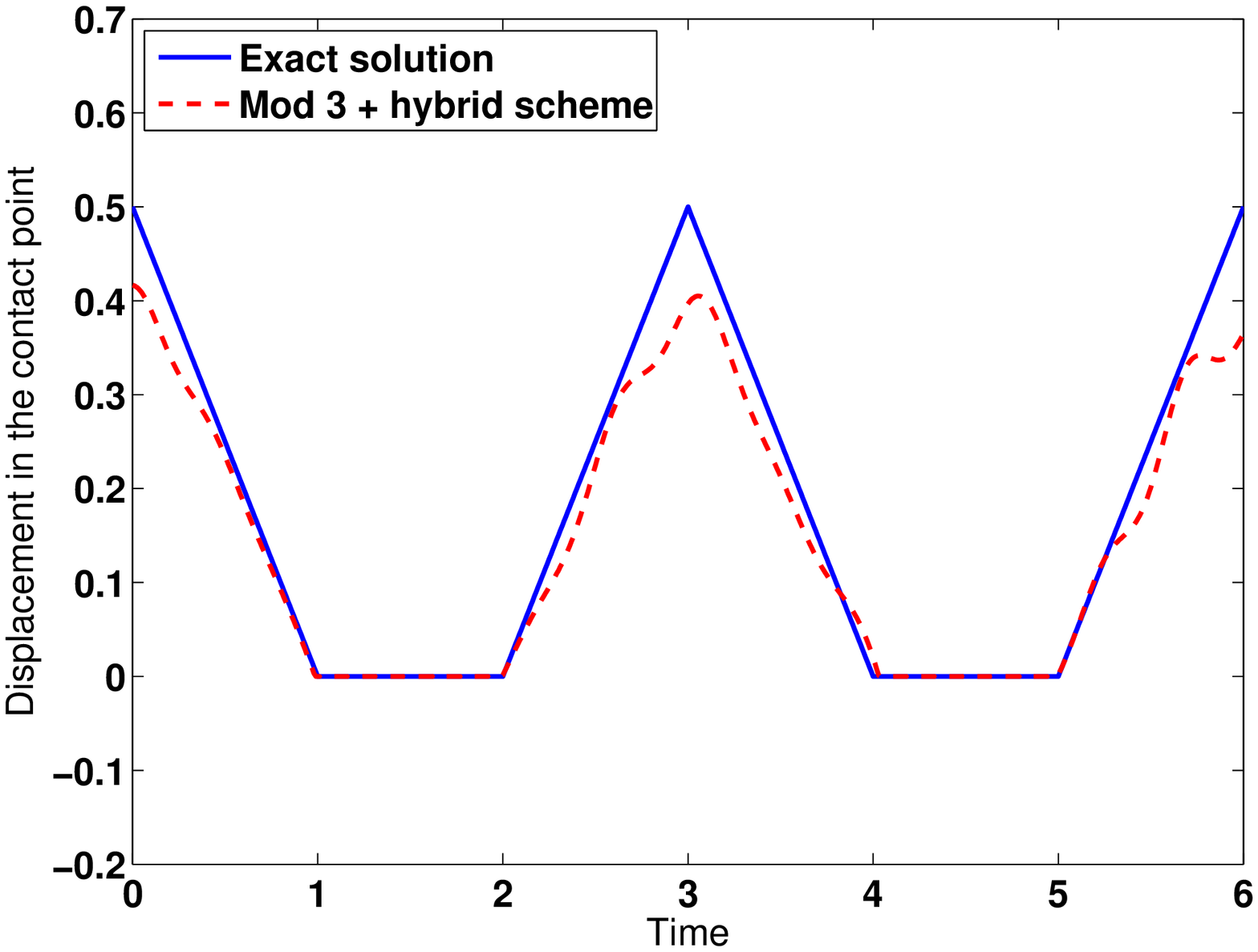}}\hspace*{-.3cm} 
       \subfigure{\includegraphics[height=5cm]{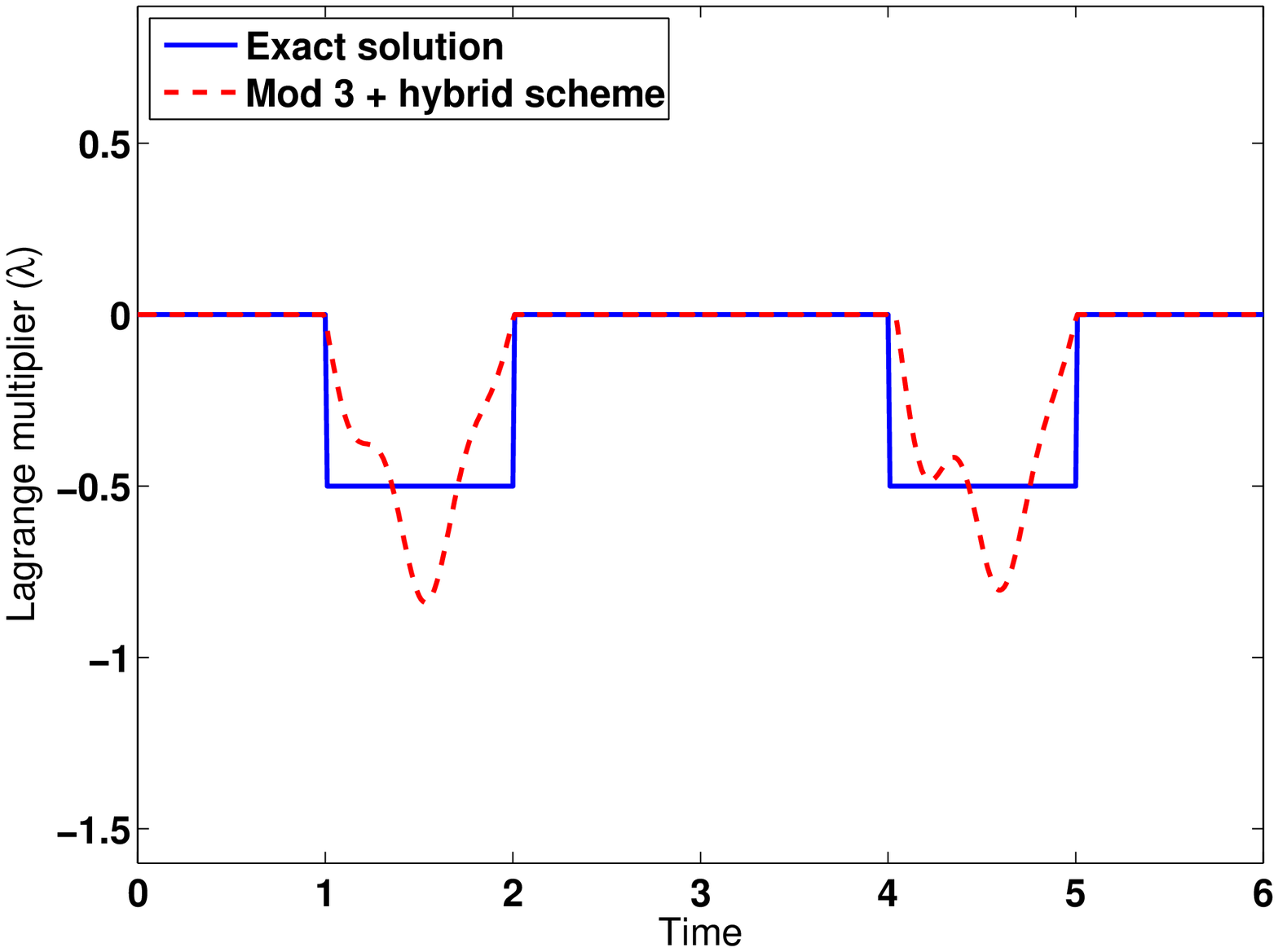}}
\caption{Comparison of analytical $(u,\lambda)$ and
approximate $(U^n_h , \lambda^n)$ solutions
for the modified mass matrix (Mod 3), in the contact node with hybrid scheme
($\Delta x = \frac{1}{6}$ and $\Delta t = \frac{1}{100}$).}\label{34} 
\end{center}   
       \end{figure}
       \end{center}
\begin{center}
 \begin{figure}[htb!]
   \begin{center}
      \subfigure{\includegraphics[height=5cm]{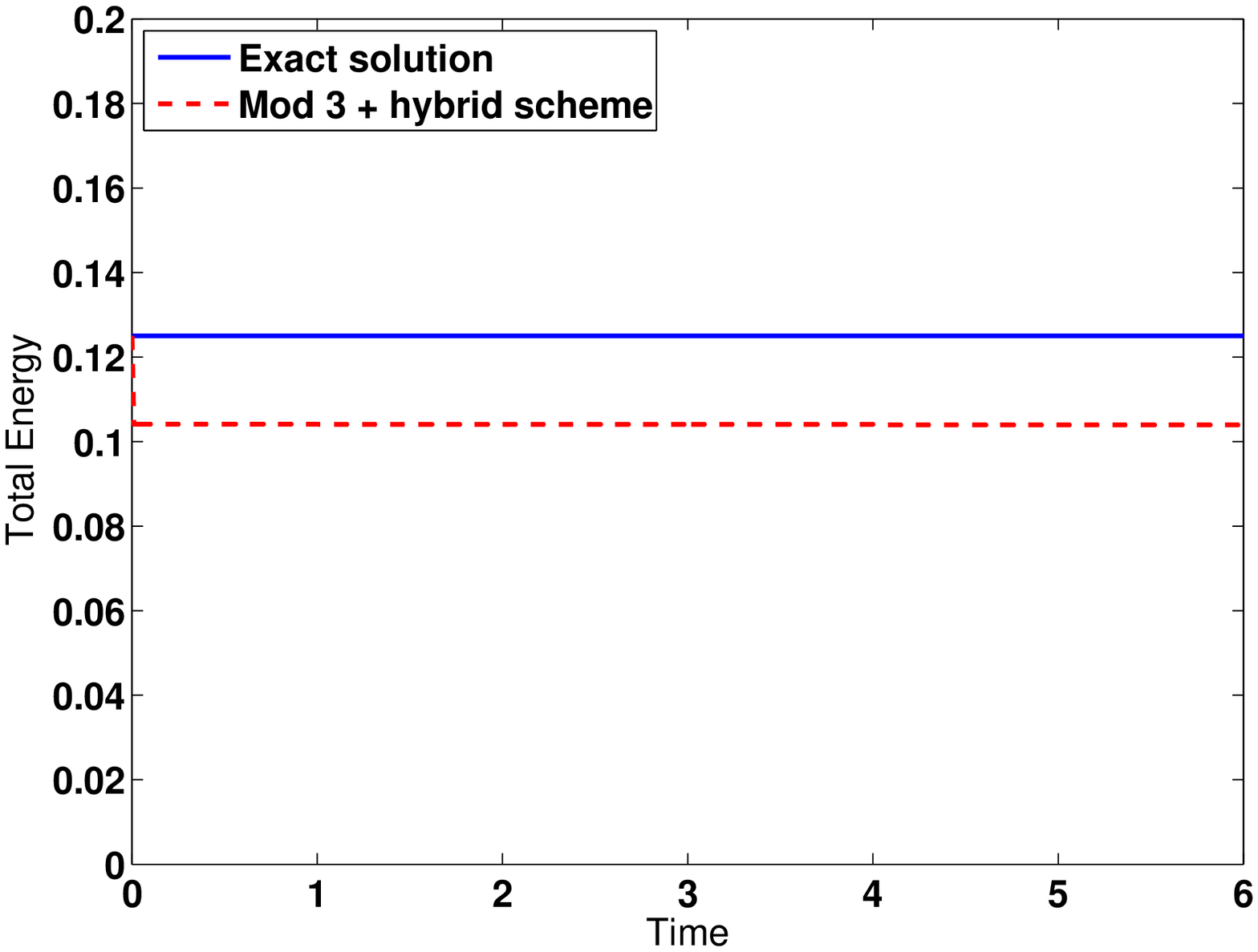}}\hspace*{-.3cm} 
       \subfigure{\includegraphics[height=5cm]{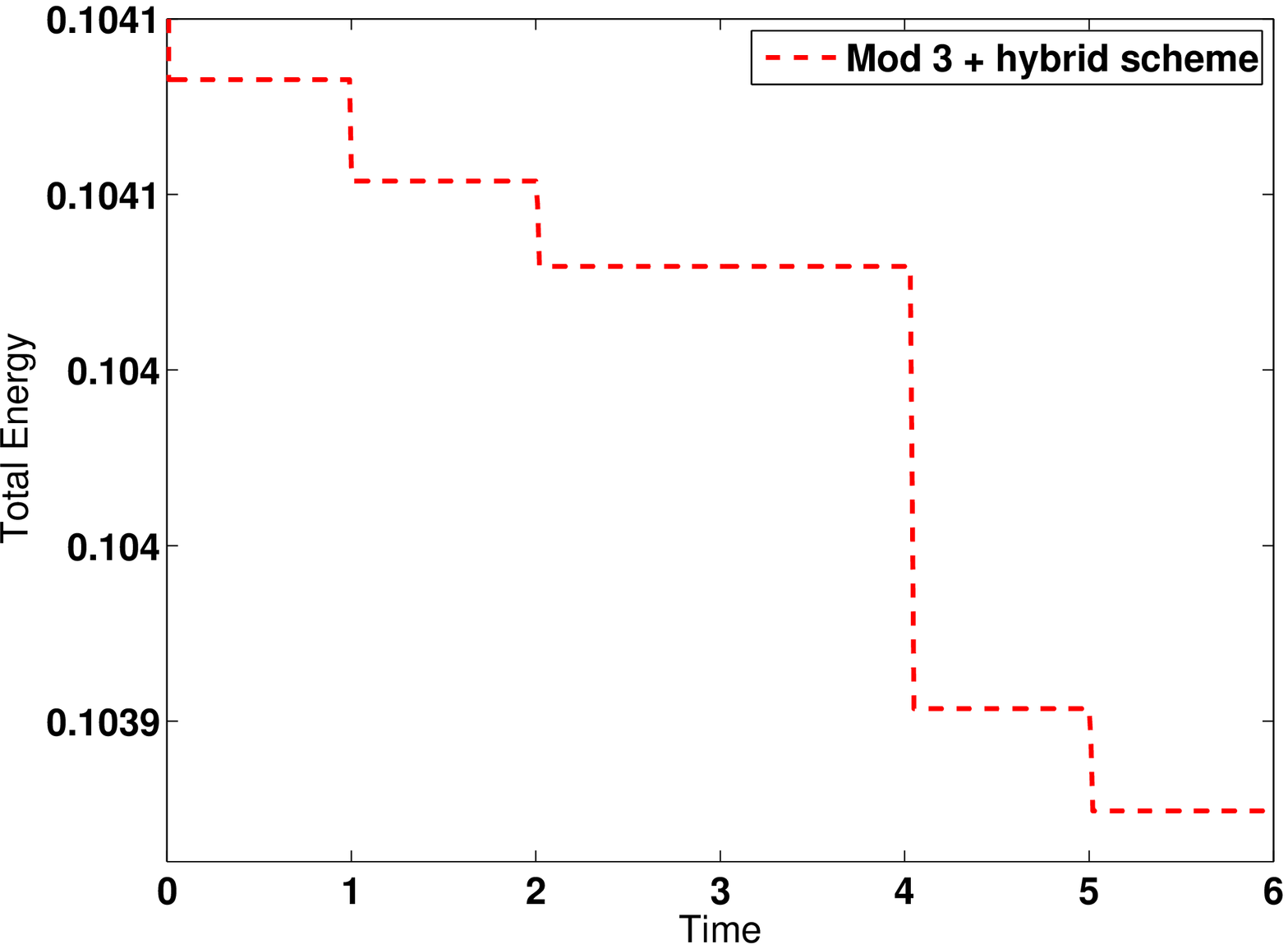}}
 \end{center}
\caption{Comparison of the energy associated with analytical solution and
the energy associated  with approximate solutions for the modified 
mass matrix (Mod 3) with hybrid scheme
($\Delta x = \frac{1}{6}$ and $\Delta t = \frac{1}{100}$). The figure on the right hand side
represents a zoom of the figure on the left hand side.}\label{35}
 \end{figure}
\end{center}
The numerical experiments presented on Figures~\ref{34} and~\ref{35}
are obtained by using a new hybrid scheme
where the mass of the contact node is redistributed on the node preceding the contact. 
This scheme circumvents the undesirable 
oscillations at the contact boundary and it prevents as well the small oscillations of
the evolution of total energy occurring for Newmark methods (see Figures~\ref{12} 
and~\ref{32}). 
Finally, the energy evolution \(\Delta\mathcal{E}_h^n\) is nonpositive for all $n>0$ and it is much
smaller than the energy evolution obtained by using implicit Euler method 
(compare with 
Figure~\ref{22}).

\section{Conclusion}
\label{conclusion}

This manuscript focuses on the weighted mass redistribution method which is particularly 
well adapted to approximate elastodynamic contact problems. 
This method leads to well-posed and energy conserving 
semi-discretization of elastodynamic contact problems. Furthermore, it 
prevents some undesirable oscillations at the contact boundary as well as 
some phase shift between approximate and analytical solutions.
The efficiency of the weighted mass redistribution method depends 
on the position of the nodes where the mass is redistributed; the closer the 
mass of the contact node is transferred, the better are the approximate solutions. 
These results seem also valid in higher space dimensions, and in particular in
2D space (see Table~\ref{tab_intro:4_ch4}).
Indeed the weight mass redistribution on the nodes before the contact 
(Mod~3) gives much better absolute error than
in the cases where any redistribution is done (Mod~1) or where the mass is just eliminated
from the contact nodes (Mod~2). 
The total error rates are
evaluated for the space steps \(\Delta x_1=\Delta x_2=0.05\).  
However the energy associated with the Newmark and Paoli-Schatzman methods
in time have small oscillations, 
(see for instance Figure~\ref{12}), which is unacceptable from a mechanical view point. 
Then a new hybrid scheme having the properties to be an 
unconditionally stable has been developed giving some promising numerical results. 
It allow to have a far better approximation compared to existing unconditionally stable
scheme on the implicit Euler scheme (compare Figures~\ref{22} and~\ref{35}).

\vspace{0.5em}
\begin {table}[!ht]
\centering
\begin{tabular}{|l|c|c|c|}
\hline
 \small{Method employed}  &  
 \small{Mod~1}  &   \small{Mod~2} &  \small{Mod~3}\\
\hline
 \small{${\norm[\L^\infty(0,T;\L^2((0,1)\times(0,1)))]{U_h^n{-}U}}$}& 
  \small{0.0104} & \small{0.0046} &  \small{0.0038}\\
\hline
\end{tabular}
\caption {Total error rates for the displacement} \label{tab_intro:4_ch4} 
\end {table}

  \renewcommand{\arraystretch}{0.91}{\small{ 
\paragraph*{Acknowledgments} 
The support of the GA\v CR Grant GA15-12227S and RVO: 67985840,
and of the AV\v CR--CNRS Project ``Mathematical and numerical analysis of 
contact problems for materials with memory" is gratefully acknowledged.

\end{document}